\documentclass{amsart}
\usepackage{amssymb,amsthm}
\usepackage{tikz}
\usetikzlibrary{arrows,matrix}
\tikzstyle{vertex} = [fill,shape=circle,node distance=80pt]
\tikzstyle{edge} = [fill,opacity=.5,fill opacity=.5,line cap=round, line join=round, line width=50pt]
\tikzstyle{elabel} =  [fill,shape=circle,node distance=30pt]
\usepackage[all]{xy}
\usepackage{float}
\usepgflibrary{patterns}
\usepackage[margin=1in]{geometry}

\numberwithin{equation}{section}

\newtheorem{thm}[equation]{Theorem}
\newtheorem{lemma}[equation]{Lemma}
\newtheorem{prop}[equation]{Proposition}
\newtheorem{cor}[equation]{Corollary}

\theoremstyle{definition}
\newtheorem{defn}[equation]{Definition}

\newtheorem{examplex}[equation]{Example}
\newenvironment{example}
  {\pushQED{\qed}\examplex}
  {\popQED\endexamplex}
\newtheorem{rmkx}[equation]{Remark}
\newenvironment{rmk}
  {\pushQED{\qed}\rmkx}
  {\popQED\endrmkx}

\DeclareSymbolFontAlphabet{\mathbb}{AMSb}

\def\m{\mathfrak{m}}

\newcommand{\dual}{\smallsmile}
 
\newcommand{\cl}{\rm cl}
\newcommand{\ri}{\rm i}
\newcommand{\cM}{\mathcal{M}}
\newcommand{\cP}{\mathcal{P}}
\newcommand{\cA}{\mathfrak{A}}
\newcommand{\ca}{\mathfrak{a}}
\newcommand{\cm}{\mathfrak{m}}
\newcommand{\cB}{\mathfrak{B}}

\newcommand{\mbf}{\cm \textrm{bf}}
\newcommand{\mbe}{\cm \textrm{be}}

\newcommand\rihull[3][\ri]{#1\textrm{-}{\rm hull}\mathnormal{^{#2}(#3)}}
\newcommand{\clcore}[3][\cl]{#1 \textrm{-}{{\rm core}}\mathnormal{_#2(#3)}}
\newcommand{\clprehull}[3][\cl]{#1\textrm{-}{\rm prehull}\mathnormal{_#2(#3)}}
\newcommand{\ripostcore}[3][\ri]{#1\textrm{-}{\rm postcore}\mathnormal{^{#2}(#3)}}
\newcommand{\clprecore}[3][\cl]{#1\textrm{-}{\rm precore}\mathnormal{_#2(#3)}}
\newcommand\riposthull[3][\ri]{#1\textrm{-}{\rm posthull}\mathnormal{^{#2}(#3)}}

\DeclareMathOperator{\core}{-core}

\title{{\cl}-prereductions, {\ri}-postexpansions, and related structures \\ }
\author[Poiani]{Sarah Poiani}
\author[Vassilev]{Janet Vassilev}
\email{spoiani@unm.edu, jvassil@unm.edu }

\address{Department of Mathematics and Statistics \\ University of New Mexico \\ Albuquerque, NM}

\begin{document}

\maketitle
\begin{abstract}
    Expanding on the work of Kemp, Ratliff and Shah \cite{KRS-prered}, for any closure $\cl$ defined on a class of modules over a Noetherian ring, we develop the theory of $\cl$-prereductions of submodules. For any interior $\ri$ on a class of $R$-modules, we also develop the theory of $\ri$-postexpansions.  Using the duality of Epstein, R.G. and Vassilev \cite{ERGV-nonres}, we show that if $\ri$ is the interior dual to $\cl$, then these notions are in fact dual to each other.   We consider the $\cl$-precore ($\ri$-postcore), the intersection of all $\cl$-prereductions ($\ri$-postexpansions) of a submodule and the $\cl$-prehull ($\ri$-posthull), the sum of all $\cl$-prereductions ($\ri$-postexpansions) of a submodule and give comparisons with the $\cl$-core ($\ri$-hull).  We further give a classification of $\cl$-prereductions of $\cl$-closed ideals of a Noetherian ring where $\cl$ is a closure with a special part.
\end{abstract}

\section{Introduction}
 
 Northcott and Rees were the first to define and study the reductions of an ideal in \cite{NR-red}.  A reduction of an ideal $I$ is an ideal $J \subseteq I$ which shares the same integral closure as $I$.  Rees further generalized the notion of integral closure and reductions to the setting of submodules of a module in \cite{reesintegralclosure}.  Generally, a submodule can have a multitude of reductions, even minimal reductions (minimal among the set of all reductions).  The core of a submodule $N$ of $M$, is the intersection of all reductions of $N$ in $M$ and is in some sense a measure of the reductions of $N$.  If $N$ is basic, i.e. $N$ is its only reduction, then ${\rm core}_M(N)=N$.  However, when $N$ has more reductions, the core is never a reduction of $N$.  In a recent work of Kemp, Ratliff and Shah \cite{KRS-prered}, the authors took a different tack by looking for the ideals contained in an ideal $I$ which does not have the same integral closure as $I$.  The maximal elements of this set, they termed {\em prereductions}.  As long as $I$ is not basic, then the core of $I$ will be contained in some prereduction of $I$.    In some sense, this is the starting place of this work.   

Moreover, as reductions have been generalized to different closure operations by Epstein in \cite{eps-spread} and \cite{nme-sp}, we will generalize prereductions to other closures as well.  We can also generalize the notion of $\cl$-core discussed in papers by Fouli and Vassilev and Vraciu \cite{FV-clcore} and \cite{FVV-tightcore} and Epstein, R.G. and Vassilev \cite{ERGV-corehull}, \cite{ERGV-nonres} to that of the $\cl$-precore of a submodule, the intersection of all $\cl$-prereductions.  One might expect that $\cl$-precore of a submodule is contained in its $\cl$-core, but this is not always the case.  We will examine some conditions on the submodule which will ensure the containment to be true as well as exhibit some counterexamples.  As computing the $\cl$-core of a submodule can be difficult and formulas for the $\cl$-core are only known in certain settings (see \cite{HS-core}, \cite{R-regloccore}, \cite{CPU-structurecore}, \cite{CPU-residualcore}, \cite{CPU-modulecore}, \cite{PB-coreformula},  \cite{HT-core},  \cite{FPU-core}, \cite{FV-clcore}, \cite{FVV-tightcore}, \cite{ERGV-nonres} and \cite{CFH}), we hope that in some instances the $\cl$-precore may be a helpful tool to give insight into computing the $\cl$-cores of submodules.

Following in the footsteps of Epstein, R.G. and Vassilev \cite{ERG-duality} \cite{ERGV-corehull}, \cite{ERGV-nonres} and \cite{ERGVresher}, we define the dual notion, $\ri$-postexpansion for an interior operation $\ri$ on a class of modules of a Noetherian local ring $(R, \m)$ and also the $\ri$-posthull.   But we also note that we can define some other interesting submodules the $\cl$-prehull and $\ri$-postcore which are sometimes comparable to the $\cl$-core and $\ri$-hull, respectively.   

 In certain cases, we can determine exactly the form of all the $\cl$-prereductions of a submodule and the form of all the $\ri$-postexpansions of a submodule.  In particular, if $N$  is a finitely generated $\cl$-basic submodule of $M$, then every $\cl$-prereduction can be determined in terms of the minimal generating sets of $N$. Similarly if $N$ is a finite length submodule in an Artinian module $M$ and $N$ is its only $\ri$-postexpansion in $M$, then the $\ri$-postexpansions can be determined in terms of the minimal cogenerating sets of $M/N$.

The paper is organized as follows. In Section 2, we recall the definitions of pair operations, closure operations, and interior operations. We also introduce reductions, expansions, cores, and hulls, including some key examples of these operations which we will make use of in this paper. In Section 3, we discuss and prove properties about $\cl$-prereductions and in Section 4 we analyze $\ri$-postexpansions. Section 5 deals with comparing $\cl$-prereductions for comparable closures and $\ri$-postexpansions for comparable interiors.  The duality between pair operations and the correspondence between $\cl$-prereductions and $\ri$-postexpansions is shown in Section 6. In Section 7, we discuss the notion of a cover of a submodule with the viewpoint of $\cl$-reductions/prereductions and $\ri$-expansions/postexpansions and determine the structure of $\cl$-prereductions  of a $\cl$-basic submodule $N$ in terms of minimal generating sets of $N$. Similarly we classify the $\ri$-postexpansions of a submodule $N$ whose only expansion is itself in terms of cogenerating sets of $M/N$. Section 8  introduces and explores the concepts of cl-prehull,  cl-precore, i-postcore, and i-posthull.  In Section 9, we exhibit that the structure of $\cl$-prereductions for $\cl$-closed ideals can also be determined when $\cl$ is a closure with special part and all strongly $\cl$-independent ideals have a special part decomposition.

\section{Background}

We start by defining pair operations; closure and interior operations are specific examples of pair operations. The notion of pair operations allows us to use a common framework for these operations as well as other operations defined on $R$-modules.

\begin{defn}
\cite[Definition 2.1]{ERGV-nonres} Let $R$ be an associative ring, not necessarily commutative. Let $\cM$ be a class of (left) $R$-modules that is closed under taking submodules and quotient modules. Let $\cP$ be the class of pairs $(N,M)$ of $R$-modules with $N,M \in \cM$ and  $N\subseteq M$.
\end{defn}

\begin{defn}
\cite[Definition 2.2]{ERGV-nonres} Let $\cP$ be a collection of pairs $(L,M)$ with $L\subseteq M$ such that whenever $\varphi:M\rightarrow M'$ is an isomorphism and $(L,M)\in\cP$ then $(\varphi(L),M')\in\cP$. A \textit{pair operation} is a function $p$ that sends each pair $(L,M)\in\cP$ to a submodule $p(L,M)$ of $M$, in such a way that whenever $\varphi:M\rightarrow M'$ is an $R$-module isomorphism and $(L,M)\in\cP$, then $\varphi(p(L,M))=p(\varphi(L),M')$. When $(L,M)\in\cP$, we say that $p$ is 
\begin{itemize}
    \item \textit{idempotent} if whenever $(L,M)\in\cP$ and $(p(L,M),M)\in\cP$, we always have $p(p(L,M),M)=p(L,M)$.
    \item \textit{order-preserving on submodules} if whenever $L\subseteq N \subseteq M$ such that when $(L,M), (N,M)\in\cP$, we have $p(L,N)\subseteq p(N,M)$.
    \item \textit{extensive} if we always have $L\subseteq p(L,M)$.
    \item \textit{intensive} if we always have $p(L,M)\subseteq L$.
    \item a \textit{closure operation} if it is extensive, idempotent, and order-preserving on submodules.  
    \item an \textit{interior operation} if it is intensive, idempotent, and order-preserving on submodules.  
\end{itemize}
\end{defn}

\begin{rmkx}
    If $p$ is a closure operation, we will denote $p(L,M)=L_M^{\cl}$ for $(L,M) \in \cP$ and refer to $p$ as $\cl$.   If $p$ is an interior operation, we will denote $p(L,M)=L^M_{\ri}$ for $(L,M) \in \cP$ and refer to $p$ as $\ri$.
We say $N$ is \textit{$\cl$-closed } in $M$ if $N=N_M^{\cl}$. We say $A$ is \textit{$\ri$-open} in $B$ if $A=A_{\ri}^B$
\end{rmkx}

Some common closure operations in commutative algebra are integral closure, tight closure and basically full closure which we will define now for ideals of a commutative Noetherian ring to use later when we present examples.  All of these closures are defined on the class of finitely generated modules over Noetherian rings.

\begin{defn}
Let $R$ be a commutative ring and $I$ an ideal of $R$.  The integral closure of $I$ is: 
\[
I^{-}:=\{x \in R \mid x^n+a_1x^{n-1}+\cdots + a_{n-1}x+a_n=0 \text{ for some } a_i \in I^i \}.
\]
\end{defn}

Note that we use $I^{-}$ instead of $\overline{I}$ or $I_a$ since we represent all closures as superscripts.  

In the late 80's, Hochster and Huneke \cite{HH1} introduced tight closure, a closure operation in equicharacteristic rings.  Here we will stick to the positive characteristic version.

\begin{defn}
Let $R$ be a Noetherian ring of characteristic $p>0$ and $I \subseteq R$ an ideal.  Set  $I^{[p^e]}=(x^{p^e} \mid x \in I)$ and $R^o=R \setminus \bigcup \{P \mid P \text{ a minimal prime of }R\}$. The tight closure of $I$ is
\[
I^*:=\{x \in R \mid cx^{p^e} \in I^{[p^e]} \text{ for some } c \in R^o \text{ and all } e>>0\}.
\]
\end{defn}

Basically full closure was introduced by Heinzer, Ratliff and Rush in \cite{HRR-bf} as a closure operation on $\m$-primary ideals.  However, the operation is also a closure operation on the set of all ideals of a ring.

\begin{defn}
Let $R$ be a commutative ring and $I$ an ideal of $R$.  The basically full closure of $I$ is $I^{\m {\rm bf}}:=(\m I: \m)$.
\end{defn}

Note that $I^* \subseteq I^-$ by \cite{HH1} and $I^{\m {\rm bf}} \subseteq I^-$ by \cite{HRR-bf}; however, tight closure and basically full closure are not comparable for general Noetherian rings.  

\begin{examplex}
Let $k$ be  a field of characteristic $p>0$.  If $R=k[[x,y]]$, $I^*=I$ for all ideals $I \subseteq R$ since $R$ is a regular ring.  However, $(x^3, y^3)^{\m {\rm bf}}=(\m (x^3,y^3): \m)=(x^3,x^2y^2,y^3)$.  So in $R$ we have $I^*=I \subseteq I^{\m {\rm bf}}$ for all ideals $I$.   

However, for $S=k[[x^2,x^5]]$, 
\[(x^4)^{\m {\rm bf}}=(\m (x^4): \m)=(x^6,x^9):(x^2,x^5)=(x^4,x^7) \subsetneq (x^4,x^5)=(x^4)^*=(x^4)^{-}\] where the last equality follows from \cite[Corollary 5.8]{HH1}.  In particular, in $S$ we have $I^{\m {\rm bf}} \subseteq I^*$ for all ideals $I$.  

Consider the ring $T=k[[x^2,x^5, y, xy]]$.  Note that $(x^4)^*=(x^4)^-=(x^4,x^5)$ by \cite[Corollary 5.8]{HH1}.  However, similar to the computation in $S$, $(x^4)^{\m {\rm bf}}=(x^4,x^7)$ and 
\[(x^4)^{\m {\rm bf}} \subseteq (x^4)^*=(x^4)^-.\]  We have included Figure~\ref{fig:latticeprin} to illustrate the monomials in $(x^4)$ which are shaded in red, the monomials in $T\setminus (x^4)$ which are shaded in blue.  The two circled lattice points on the $x$-axis indicate the monomials which are in $(x^4)^* \setminus (x^4)$ where the darker blue lattice point is a monomial in $(x^4)^{\m {\rm bf}}$ not in $(x^4)$.

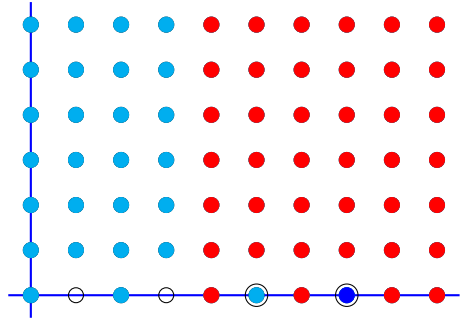
\begin{figure}
\centering
\begin{tikzpicture}[scale=0.6]
\draw[blue, thick] (-.5,0) -- (9.5,0);
\draw[blue, thick] (0,-.5) -- (0,6.5);
\foreach \x in {0,1,...,9}{
     \foreach \y in {0,1,...,6}{
       \node[draw,circle,inner sep=2pt] at (\x,\y) {};}}
\foreach \x in {1,2,...,9}{
     \foreach \y in {1,2,...,6}{
       \node[draw,circle,inner sep=2pt,cyan,fill] at (\x,\y) {};}}
 \foreach \y in {0,1,...,6}{
 \node[draw,circle,inner sep=2pt,cyan,fill] at (0,\y) {};}
\foreach \x in {2,4,5,...,9}{
 \node[draw,circle,inner sep=2pt,cyan,fill] at (\x,0) {};}
\foreach \x in {4,5,...,9}{
     \foreach \y in {1,2,...,6}{
       \node[draw,circle,inner sep=2pt,red,fill] at (\x,\y) {};}}
\foreach \x in {4,6,8,9}{
 \node[draw,circle,inner sep=2pt,red,fill] at (\x,0) {};}
\foreach \x in {5,7}{
 \node[draw,circle,inner sep=3pt] at (\x,0) {};}
\foreach \x in {7}{
 \node[draw,circle,inner sep=2pt,blue,fill] at (\x,0) {};}
\end{tikzpicture}
\caption{\ Lattice of monomials in $(x^4) \subseteq (x^4)^{\m {\rm bf}} \subseteq (x^4)^* \subseteq T$.}
\label{fig:latticeprin}
\end{figure}

Whereas, for the ideal $J=(x^4,x^5,y^2, xy^2)$: we will see that
$J^*=J\subseteq J^{\m {\rm bf}}=J+(x^2y,x^3y)$.  The first equality holds by \cite[Lemma 4.11]{HH1}, since $T \subseteq k[[x,y]]$ which is a regular ring and the preimage of $(x^4,y^2)$ in $T$ is precisely $J$.
Figure \ref{fig:latticeJ}, helps us to understand how we obtain the $\m$-basically full closure $J$.  Note the red lattice points indicate monomials in $J$ and the blue lattice points indicate the monomials in $T \setminus J$.  The lattice points at $(1,0)$ and $(3,0)$ are not colored because the monomials $x$ and $x^3$ are not in $T$. The circled red lattice points are those in $\m J$. Note that every monomial in the maximal ideal multiplies the two monomials represented by the darker blue lattice points ($x^2y$ and $x^3y$) whereas for each of the remaining monomials represented by the light blue lattice points, there is at least one element of the maximal ideal that does not multiply the monomial into $\m J$.  Hence $J^{\m {\rm bf}}=J+(x^2y,x^3y)$.

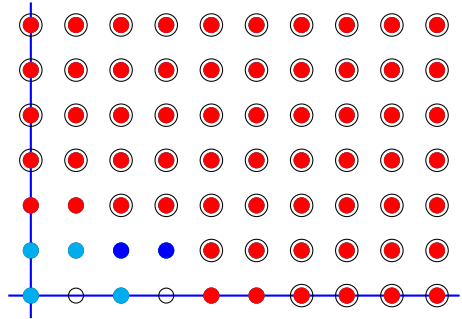
\begin{figure}[h]
\centering
\begin{tikzpicture}[scale=0.6]
\draw[blue, thick] (-.5,0) -- (9.5,0);
\draw[blue, thick] (0,-.5) -- (0,6.5);
\foreach \x in {0,1,...,9}{
     \foreach \y in {0,1,...,6}{
       \node[draw,circle,inner sep=2pt] at (\x,\y) {};}}
 \foreach \x in {2,3}{
       \node[draw,circle,inner sep=3pt] at (\x,2) {};}
\foreach \x in {4,5,...,9}{
     \foreach \y in {1,2,...,6}{
       \node[draw,circle,inner sep=3pt] at (\x,\y) {};}}
\foreach \x in {0,1,...,3}{
     \foreach \y in {3,4,...,6}{
       \node[draw,circle,inner sep=3pt] at (\x,\y) {};}}
\foreach \x in {1,2,...,9}{
     \foreach \y in {1,2,...,6}{
       \node[draw,circle,inner sep=2pt,cyan,fill] at (\x,\y) {};}}
 \foreach \y in {0,1,...,6}{
 \node[draw,circle,inner sep=2pt,cyan,fill] at (0,\y) {};}
\foreach \x in {2,4,5,...,9}{
 \node[draw,circle,inner sep=2pt,cyan,fill] at (\x,0) {};}
\foreach \x in {6,7,...,9}{
 \node[draw,circle,inner sep=3pt] at (\x,0) {};}
\foreach \x in {4,5,...,9}{
     \foreach \y in {1,2,...,6}{
       \node[draw,circle,inner sep=2pt,red,fill] at (\x,\y) {};}}
\foreach \x in {1,2,3}{
     \foreach \y in {2,3,...,6}{
       \node[draw,circle,inner sep=2pt,red,fill] at (\x,\y) {};}}
\foreach \y in {2,3,...,6}{
 \node[draw,circle,inner sep=2pt,red,fill] at (0,\y) {};}
\foreach \x in {4,5,...,9}{
 \node[draw,circle,inner sep=2pt,red,fill] at (\x,0) {};}
\foreach \x in {2,3}{
 \node[draw,circle,inner sep=2pt,blue,fill] at (\x,1) {};}
\end{tikzpicture}
\caption{\ Lattice of monomials in $J \subseteq J^{\m {\rm bf}} \subseteq T$.}
\label{fig:latticeJ}
\end{figure}

Since we have obtained two ideals $I$ in $T$ one with $I^{\m {\rm bf}} \subseteq I^*$ and the other with $I^* \subseteq I^{\m {\rm bf}}$, we see that tight closure and $\m$-basically full closure are not comparable in $T$.\qed
\end{examplex}

Northcott and Rees  \cite{NR-red} defined an ideal $J \subseteq I$ to be a reduction of $I$ if there exists some non-negative integer $n$ with $JI^n=I^{n+1}$ and they showed that $J$ is a reduction of $I$ if and only if $J^{-}=I^{-}$.  Rees generalized the notion of reduction for submodules of modules in \cite{reesintegralclosure}.  Then Epstein generalized reductions of ideals of a commutative Notherian ring and submodules of finitely generated modules for closure operations $\cl$ in   \cite{eps-spread} and \cite{nme-sp}.  We include the definition below in the language of pair operations as well as the related notion of $\cl\core$ which was originally introduced for integral closure by Rees and Sally in \cite{RS-reductions} and then by Fouli and Vassilev in \cite{FV-clcore} for more general closure operations $\cl$.

\begin{defn}
\cite[Definition 2.10]{ERGV-nonres} Let $R$ be a Noetherian ring, $\cM$ be a class of $R$-modules and $\cl$ be a closure operation defined on $\cP$ of pairs of modules $(L,M)$ with $L \subseteq M$ in $\cM$.  Suppose $(L,M), (N,M) \in \cP$.
\begin{enumerate}
\item   We say that $L$ is a \textit{$\cl$-reduction of $N$ in $M$}  if $L\subseteq N \subseteq L_M^{\cl}$. 
\item If $L$ is a $\cl$-reduction of $N$ in $M$ and there is no submodule $K\subseteq L$, with $(K,M) \in \cP$ such that $K_M^{\cl}=N_M^{\cl}$ then we say that $L$ is a \textit{minimal} $\cl$-reduction of $N$ in $M$.
\item We define ${\cl\core}_M(N)= \bigcap \{ L \mid L \subseteq N \subseteq L_M^{\cl} \text{ and } (L,M) \in \cP\}$.
\end{enumerate}
\end{defn}

Although we defined minimal $\cl$-reductions of a submodule $N$ of $M$ above, we do not in general know that minimal $\cl$-reductions exist.  If they do, then  
\[{\cl\core}_M(N)=\bigcap \{ L \mid L \text{ a minimal {\cl}-reduction of } N \text{ and } (L,M) \in \cP\}.\]

\begin{defn}
   \cite[Definition 1.2]{eps-spread} Let $(R, \m)$ be a Noetherian local ring and $\cl$ be a closure operation on the class of pairs $\cP$ of finitely generated $R$-modules. We say that $\cl$ is a \textit{Nakayama closure} if for $L \subseteq N \subseteq M$ finitely generated $R$-modules, if $L \subseteq N \subseteq (L +\m N)^{\cl}_M$ then $L^{\cl}_M=N^{\cl}_M$.
\end{defn}

If $\cl$ is a Nakayama closure on the class of finitely generated modules, then Epstein showed that minimal $\cl$-reductions exist first for ideals in \cite[Lemma 2.2]{eps-spread} and then noted that they also exist for submodules of finitely generated modules \cite[Section 1]{nme-sp}.

In a recent paper of Kemp, Ratliff and Shah \cite{KRS-prered}, the authors define an ideal $J \subseteq I$ to be a prereduction if $J$ is not an (integral) reduction of $I$ however for all $K$ with $J \subseteq K \subseteq I$, $K$ is an (integral) reduction of $I$.  They also consider the set of ideals ${\bf I}'(I)=\{J\subseteq I  \mid J \text{ is not an (integral) reduction of } I\}$.  They note that if ${\bf I}'(I)$ is non-empty then the maximal elements of ${\bf I}'(I)$ are prereductions of $I$.  Minimal reductions and the analytic spread of an ideal are important to their work.  The analytic spread of an ideal $I$ is the maximal number of algebraically independent elements in $I$.    The following is a generalization of algebraic independence inspired by Vraciu's work on special tight closure and $*$-independence in \cite{Vr-*ind} given by Epstein \cite{eps-spread} and \cite{nme-sp}.

\begin{defn}
Let $R$ be a Noetherian ring and $\cl$ be a closure operation defined on $R$-modules.  We say that $f_1, \ldots, f_r \in M$ are $\cl$-independent if $f_i \notin (f_1R+ \cdots + \hat{f}_i R+ \cdots+ f_rR)^{\cl}$.  We say a submodule  $N \subseteq M$ is {\it strongly $\cl$-independent} if every minimal set of generators of $N$ is $\cl$-independent.
\end{defn}

Much of Kemp, Ratliff and Shah's work is over a Noetherian local ring and integral closure is a Nakayama closure.
Epstein proved that when $\cl$ is a Nakayama closure an ideal $L \subseteq N$ is a reduction of $N$ in $M$, then $L$ is minimal $\cl$-reduction of $N$ if and only if $L$ is a strongly $\cl$ independent.  He further generalized the notion of analytic spread to Nakayama closures $\cl$.

\begin{defn} (See \cite{eps-spread}) Let $(R,\m)$ be a Noetherian local ring and $\cl$ a Nakayama closure defined on $R$-modules.  We say $N$ has {\it $\cl$-spread} if the cardinality of any minimal generating set for any minimal reduction of $N$ is the same.
\end{defn}

\section{$\cl$-prereductions}

For any Nakayama closure and any submodule $N \subseteq M$ we can clearly define a set of submodules of $M$ whose closure is properly contained in the closure of $N$. As Kemp, Ratliff and Shah denoted such a set ${\bf I}'(I)$ for integral closure for ideals $I$ of $R$, we will modify their notation including the closure operation $\cl$ in the subscript and the pair of modules $(N,M) \in \cP$ to define \[{\bf I}'_{\cl}(N,M):=\{L \subseteq N \mid L \text{ is not a } {\cl} \text{-reduction of } N\}.\]  Note that for any $L \subseteq N$, then  $L \notin {\bf I}'_{\cl}(N,M)$ if and only if $L$ is a $\cl$-reduction of $N$ in $M$.  If $M=R$, we will omit $R$ and denote ${\bf I}'_{\cl}(I,R)={\bf I}'_{{\cl}}(I)$.

It may be the case that ${\bf I}'_{\cl}(N,M)$ is empty.  For example if $N=(0) \subseteq M$, then $(0)$ is a $\cl$-reduction of itself and contains no proper submodules so ${\bf I}'_{\cl}((0),M)=\emptyset$.   In \cite[Remark 3.5.1]{KRS-prered}, Kemp, Ratliff and Shah note that  ${\bf I}'(I)$ is nonempty if and only if $I$ is not a nilpotent ideal.  However, unlike in the case for integral closure of ideals in a ring,  ${\bf I}'_{\cl}(I)$ may not be empty when $I$ is a nilpotent ideal.  For example:

\begin{examplex}
\label{nonemptyI}{\rm
Let   $R=k[[x,y]]/(x^2y^2)$ and $I=(xy)$.  Note that $(xy)$ is the nilradical of $R$ and a nilpotent ideal.  Note that $(0)^{\m {\rm bf}}=((0):\m)=(0)$ and \[(xy)^{\m {\rm bf}}=(\m (xy):\m)=(x^2y,xy^2):(x,y)=(xy).\]  We see that ${\bf I}'_{\m {\rm bf}}(xy) \neq \emptyset$ even though it is a nilpotent ideal.
}\end{examplex}

The order operation {\rm ord}, defined by ${\rm ord}(I)=\m^{r}$ if $I \subseteq \m^r$ but $I \nsubseteq \m^n$ for any $n>r$ and ${\rm ord}(I)=\bigcap\limits_{r \in \mathbb{N}} \m^r$ if $I \subseteq \m^r$ for all $r \in \mathbb{N}$ as discussed in \cite{Va_naksem} is also a Nakayama closure with $(0) \in {\bf I}'_{\rm ord}(xy)$ in $R$ as in Example \ref{nonemptyI} as ord$(xy)=\m^2$ and $(0)={\rm ord}(0)$.  

Although, we have given examples above of Nakayama closure operations where there are nilpotent ideals $I$ with ${\bf I}'_{\cl}(I) \neq \emptyset$, tight closure behaves more like integral closure for ideals in the sense that ${\bf I}'_{*}(I)=\emptyset$ when $I$ is nilpotent.

\begin{prop} \label{prop:niltight}
Let $(R,\m)$ be a Noetherian local ring of characteristic $p>0$, then  for any nilpotent ideal $I$,  ${\bf I}'_{*}(I)=\emptyset$.
\end{prop}

\begin{proof}  Denote the nilradical by $N$. For any nilpotent ideal $I$ we have $(0) \subseteq I \subseteq N$.   By \cite[Proposition 4.1(i)]{HH1}, $(0)^*=N=I^*$.  
For any ideal $J \subseteq I$ we have  $(0) \subseteq J \subseteq I$; hence it is clear that $I$ has no ideals $J$ with $J \subseteq I$ and $J^* \subsetneq I^*$.
\end{proof}

Because multiplication of elements in a module is not defined unless we extend multiplication through the tensor product; thus, we do not usually discuss nilpotent submodules of an $R$-module.

Due to the above examples and comment, we see that not all properties that Kemp, Ratliff and Shah obtained for ${\bf I}'(I)$ in \cite{KRS-prered} generalize for ${\bf I}_{\cl}'(N,M)$ for a general  Nakayama closure $\cl$. In addition to defining the set ${\bf I}'_{\cl}(N,M)$, we can also define the notion of \cl-prereductions for general Nakayama closure operations $\cl$.

\begin{defn}
Let $(R,\m)$ be a Noetherian local ring.  We say that $L$ is a \textit{$\cl$-prereduction} of $N$ if $L\subseteq N$, $L$ is not a $\cl$-reduction of $N$ and for all $K$ with $L \subseteq K \subseteq N$, $K$ is a $\cl$-reduction of $N$.
\end{defn}

 If ${\bf I}'_{\cl}(N,M) \neq \emptyset$, the maximal elements of ${\bf I}'_{\cl}(N,M)$ are $\cl$-prereductions.  We will denote the set of $\cl$-prereductions by ${\bf P}_{\cl}(N,M)$.
Note that if ${\bf I}'_{\cl}(N,M)$ is the set of all ideals properly contained in $N$, then $N$ has no $\cl$-reductions.  We usually call such a submodule $N$ of $M$ which only has  itself as a $\cl$-reduction {\it $\cl$-basic}.  

Instead of the  non-nilpotence assumption as Kemp, Ratliff and Shah do for \cite[Proposition 3.7]{KRS-prered}, we will require ${\bf I}'_{\cl}(N,M)$ to be nonempty for our generalization.

\begin{prop} \label{pr:nonreds}
Let $(R,\cm)$ be a Noetherian local ring and {\cl} a Nakayama closure on $\cP$, the class of finitely generated $R$-modules. Suppose $(N,M)\in \cP$ such that ${\bf I}_{\cl}^\prime(N,M)\neq\emptyset$ then the following hold
\begin{enumerate}
    \item \label{it:1nonred} Suppose $K\in {\bf I}_{\cl}^\prime(N,M)$ and $L \subseteq K$ then $L\in {\bf I}_{\cl}^\prime(N,M)$.
    \item \label{it:2nonred} Let $L\in {\bf I}_{\cl}^\prime(N,M)$. There exists a submodule $\cA\in {\bf I}_{\cl}^\prime(N,M)$ which is maximal in ${\bf I}_{\cl}^\prime(N,M)$ and $\cA\supseteq L$.
    \item \label{it:3nonred} Suppose that $\cA_1$ and $\cA_2$ are both maximal submodules in ${\bf I}_{\cl}^\prime(N,M)$. Then $\cA_1+\cA_2 \notin {\bf I}_{\cl}^\prime(N,M)$ and $\cA_1+\cA_2$ is a {\cl}-reduction of $N$ in $M$.
    \item \label{it:4nonred} If $L,K\in {\bf I}_{\cl}^\prime(N,M)$ then either $L+K \in {\bf I}_{\cl}^\prime(N,M)$ or $L+K$ is a {\cl}-reduction of $N$ in $M$.
    \item \label{it:5nonred} If $L\in {\bf I}_{\cl}^\prime(N,M)$ then $L+\cm N\in {\bf I}_{\cl}^\prime(N,M)$.
    \item \label{it:6nonred} If $L\in {\bf I}_{\cl}^\prime(N,M)$ then $L_M^{\cl}\cap N\in {\bf I}_{\cl}^\prime(N,M)$.
    \item \label{it:7nonred} If $\cA$ is maximal in ${\bf I}_{\cl}^\prime(N,M)$ then $(\cA_M^{\cl}+(\cm N)_M^{\cl})_M^{\cl}\cap N=\cA$.
\end{enumerate}
\end{prop}

\begin{proof}

 \eqref{it:1nonred} If $K\in {\bf I}_{\cl}^\prime(N,M)$, then $K$ is not a {\cl}-reduction of $N$. Since $L\subseteq K$ and $L_M^{\cl} \subseteq K_M^{\cl} \subsetneq N_M^{\cl}$ then $L$ is also not a {\cl}-reduction of $N$ in $M$.
 
    \eqref{it:2nonred} Since $L\in {\bf I}_{\cl}^\prime(N,M)$ and $R$ Noetherian, then there exists an element $\cA\in {\bf I}_{\cl}^\prime(N,M)$ which is maximal in ${\bf I}_{\cl}^\prime(N,M)$ and $\cA\supseteq L$.
    
    \eqref{it:3nonred} Since $\cA_1$ and $\cA_2$ are both maximal in ${\bf I}_{\cl}^\prime(N,M)$ and $\cA_1,\cA_2\subseteq \cA_1+\cA_2 \subseteq N$, then $\cA_1+\cA_2$ is a {\cl}-reduction of $N$ in $M$.
    
    \eqref{it:4nonred} Since $L,K\subseteq N$ then $L+K\subseteq N$.  If $L+K \in {\bf I}_{\cl}^\prime(N,M)$, we are done. If $L+K \notin {\bf I}_{\cl}^\prime(N,M)$, then by definition $L+K$ is a {\cl}-reduction of $N$ in $M$.
    
    \eqref{it:5nonred} Suppose $L+\cm N\notin {\bf I}_{\cl}^\prime(N,M)$. So $L+\cm N$ is a {\cl}-reduction of $N$ in $M$. Then $L\subseteq N\subseteq (L+\cm N)_M^{\cl}$. Because {\cl} is a Nakayama closure, $L_M^{\cl}=N_M^{\cl}$ and $L$ is a {\cl}-reduction of $N$ in $M$ which contradicts our assumption. So $L+\cm N\in {\bf I}_{\cl}^\prime(N,M)$.
    
    \eqref{it:6nonred} Suppose that $L_M^{\cl}\cap N\notin {\bf I}_{\cl}^\prime(N,M)$. So $L_M^{\cl}\cap N$ is a {\cl}-reduction of $N$ in $M$. Then $L_M^{\cl}\subseteq N_M^{\cl} = (L_M^{\cl}\cap N_M^{\cl})_M^{\cl} \subseteq (L_M^{\cl})_M^{\cl}=L_M^{\cl}$ and $L$ is a {\cl}-reduction of $N$ in $M$ which contradicts $L\in {\bf I}_{\cl}^\prime(N,M)$.
    
    \eqref{it:7nonred} Since $(\cA+\cm N)\subseteq (\cA_M^{\cl}+(\cm N)_M^{\cl})\subseteq (\cA+\cm N)_M^{\cl}$, we have $(\cA_M^{\cl}+(\cm N)_M^{\cl})_M^{\cl}=(\cA+\cm N)_M^{\cl}$. By (5), we know that $\cA+\cm N\in {\bf I}_{\cl}^\prime(N,M)$. Since $\cA\subseteq \cA+\cm N$ and $\cA$ maximal, then $\cA=\cA+\cm N$ and hence $\cA_M^{\cl}=(\cA_M^{\cl}+(\cm N)_M^{\cl})_M^{\cl}$. By (6), we know that $(\cA_M^{\cl}+(\cm N)_M^{\cl})_M^{\cl}\cap N \in {\bf I}_{\cl}^\prime(N,M)$ and by the maximality of $\cA$ we have $\cA=(\cA_M^{\cl}+(\cm N)_M^{\cl})_M^{\cl}\cap N$.
\end{proof}

The following proposition gives us some nice properties that always hold for $\cl$-prereductions.

\begin{prop} \label{pr:prered}
Let $(R,\cm)$ be a Noetherian local ring and $\cl$ a Nakayama closure on $\cP$. Then
\begin{enumerate}
    \item \label{it:1prered} Every submodule $L\subseteq N$ which is not a $\cl$-reduction of $N$ in $M$ is contained in a $\cl$-prereduction of $N$ in $M$.
    \item  If $\cA$ is a $\cl$-prereduction of $N$ in $M$. Then 
    \begin{enumerate}
        \item \label{it:2aprered} $\cm N \subseteq \cA$.
        \item \label{it:2bprered} $\cA$ is $\cl$-closed in $N$ or $\cA_M^{\cl}\cap N=\cA$.
    \end{enumerate}
\end{enumerate}
\end{prop}

\begin{proof}
\begin{enumerate}
    \item By Proposition \ref{pr:nonreds}\eqref{it:2nonred}, there is some maximal element $\cA$ of ${\bf I}_{\cl}^\prime(N,M)$ which contains $L$. Such an $\cA$ must be a $\cl$-prereduction since any submodule containing it must be a $\cl$-reduction of $N$ in $M$.
    \item \begin{enumerate}
        \item In the proof of Proposition \ref{pr:nonreds}\eqref{it:7nonred}, we saw $\cA=\cA+\cm N$. This implies $\cm N \subseteq\cA$.
        \item We also saw in the proof of Proposition \ref{pr:nonreds}\eqref{it:7nonred} that $\cA_M^{\cl}\cap N= (\cA+\cm N)_M^{\cl} \cap N =\cA$.
    \end{enumerate}
\end{enumerate}
\end{proof}

In particular, if a submodule is $\cl$-closed in $M$ then the $\cl$-prereductions will also be $\cl$-closed.

\begin{cor}\label{co:closedprered}
 Let $(R,\cm)$ be a Noetherian local ring and $\cl$ a Nakayama closure on $\cP$. For every $\cl$-prereduction $\cA$ of $N_M^{\cl}$, $\cm N_M^{\cl}\subseteq \cA$ and $\cA=\cA_M^{\cl}$.
\end{cor}

\begin{proof}
Note that $\cA$ is also a $\cl$-prereduction of $N_M^{\cl}$ since $\cA\in N_M^{\cl}$ and every submodule $L$ with $\cA\subseteq L\subseteq N$, $L$ is a $\cl$-prereduction of $N$ and hence a $\cl$-reduction of $N_M^{\cl}$. By Proposition \ref{pr:prered}\eqref{it:2aprered}, $\cm N_M^{\cl}\subseteq \cA$ and by Proposition \ref{pr:prered}\eqref{it:2bprered}, $\cA_M^{\cl}=\cA_M^{\cl}\cap N_M^{\cl}=\cA$.
\end{proof}

\begin{prop} \label{pr:preredcompare}
 Let $(R,\cm)$ be a Noetherian local ring and $\cl$ a Nakayama closure on the modules of $R$. Let $K\subseteq N \subseteq M$ be submodules of $R$ with $K$ a $\cl$-reduction of $N$ in $M$. Then
\begin{enumerate}
    \item \label{it:preredcomp1} ${\bf I}_{\cl}^\prime(K,M)\subseteq {\bf I}_{\cl}^\prime(N,M)$.
    \item \label{it:preredcomp2} For each $L\in {\bf I}_{\cl}^\prime(N,M)$, $L\cap K \in {\bf I}_{\cl}^\prime(K,M)$.
    \item \label{it:preredcomp3} For each maximal element of $\cA$ of ${\bf I}_{\cl}^\prime(K,M)$ there exists a maximal element $\cB$ of ${\bf I}_{\cl}^\prime(N,M)$ such that $\cB\cap K=\cA$.
\end{enumerate}
\end{prop}

\begin{proof}

    \eqref{it:preredcomp1} Let $L\in {\bf I}_{\cl}^\prime(K,M)$. Since $K\subseteq N$, then $L\subseteq K \subseteq N$ and $L_M^{\cl} \subsetneq K_M^{\cl} \subseteq N_M^{\cl}$. Since $L$ is not a $\cl$-reduction of $K$, it cannot be a $\cl$-reduction of the larger module $N$. Thus $L\in {\bf I}_{\cl}^\prime(N,M)$.
    
    \eqref{it:preredcomp2} If $L\in {\bf I}_{\cl}^\prime(N,M)$, then $L\subseteq N$ and $L$ is not a $\cl$-reduction of $N$. Note that $L\cap K \subseteq K$. To see that $L\cap K\in {\bf I}_{\cl}^\prime(K,M)$, it is enough to see that $L\cap K$ is not a $\cl$-reduction of $K$. Suppose that $L\cap K$ is a $\cl$-reduction of $K$. Then $(L\cap K)_M^{\cl}=K_M^{\cl}$. Note that $(L\cap K)_M^{\cl} \subseteq L_M^{\cl} \cap K_M^{\cl}$. Since $K$ is a $\cl$-reduction of $N$, then $N_M^{\cl}=K_M^{\cl}\subseteq L_M^{\cl}\subseteq N_M^{\cl}$ which gives a contradiction to $L\in {\bf I}_{\cl}^\prime(N,M)$. Hence, $L\cap K \in {\bf I}_{\cl}^\prime(K,M)$.
    
    \eqref{it:preredcomp3} Let $\cA$ be a maximal element of ${\bf I}_{\cl}^\prime(K,M)$. By (1), ${\bf I}_{\cl}^\prime(K,M) \subseteq {\bf I}_{\cl}^\prime(N,M)$. Thus $\cA\in {\bf I}_{\cl}^\prime(N,M)$ and there must exist a maximal element $\cB\in {\bf I}_{\cl}^\prime(N,M)$ with $\cA\subseteq\cB$. By (2), $\cB\cap K \in {\bf I}_{\cl}^\prime(K,M)$. Since $\cA\subseteq\cB\cap K$ and $\cA$ is maximal, we get $\cA=\cB\cap K$.

\end{proof}

\begin{cor}\label{cor:lyingoverprered}
Let $(R,\m)$ be a Noetherian local ring and $\cl$ a Nakayama closure on $\cP$, pairs of finite $R$-modules.  Let $L\subseteq N$ be submodules of $M$ with $L$ a $\cl$-reduction of $N$ in $M$.  If $\cA$ is a $\cl$-prereduction of $L$ then there exists a $\cl$-prereduction $\cB$ of $N$ with $\cB \cap L=\cA$.
\end{cor}

\begin{proof}
This is a direct consequence of Propostion \ref{pr:preredcompare}\eqref{it:preredcomp3} and the fact that maximal elements of ${\bf I}'_{\cl}(N,M)$ are $\cl$-prereductions of $N$ for any submodule $N \subseteq M$.
\end{proof}

\begin{prop} \label{pr:preredcontainment}
 Let $(R,\cm)$ be a Noetherian local ring and $\cl$ a Nakayama closure on the modules of $R$. If $\cA$ is a $\cl$-prereduction of $N$ in $M$ and $\cA\subseteq K \subseteq N$, then $\cA$ is a $\cl$-prereduction of $K$ in $M$.
\end{prop}

\begin{proof}
Since $\cA$ is a $\cl$-prereduction of $N$ in $M$ and $\cA\subseteq K \subseteq N$, then $K$ is a $\cl$-reduction of $N$ in $M$. Also $\cA\in {\bf I}_{\cl}^\prime(N,M)$. By Proposition \ref{pr:preredcompare}\eqref{it:preredcomp2}, $\cA\cap K=\cA\in {\bf I}_{\cl}^\prime(N,M)$.

Suppose that $\cA$ is not a $\cl$-prereduction of $K$ in $M$. Then there exists a maximal $\cB\in {\bf I}_{\cl}^\prime(K,M)$ with $\cA\subsetneq\cB$ and $\cB$ is a $\cl$-prereduction of $K$ in $M$. Then by Proposition \ref{pr:preredcompare}\eqref{it:preredcomp3} and since maixmal elements of ${\bf I}_{\cl}^\prime(N,M)$ are $\cl$-prereductions of $N$ in $M$ for any module $N$, there exists a $\cl$-prereduction $\mathfrak{C}$ of $N$ such that $\mathfrak{C}\supseteq \mathfrak{C}\cap K = \cB\supsetneq \cA$. This contradicts the maximality of $\cA$ in ${\bf I}_{\cl}^\prime(N,M)$. Hence, $\cA$ is a $\cl$-prereduction of $K$ in $M$.
\end{proof}

For submodules which do not have cyclic $\cl$-prereductions, we obtain the following which is similar to \cite[Proposition 3.13]{KRS-prered} for ideals with no principal prereductions.

\begin{prop}\label{prop:unionprered}
Let $(R,\m)$ be a Noetherian local ring and $\cl$ a Nakayama closure on the modules of $R$.  Let $N \subseteq M$, be a submodule. Then $N = \cup \{\ca \mid \ca \text{ a } $\cl$ \text{-prereduction of } N\}$ if and only if $N$ has no cyclic $\cl$-reductions in $M$.
\end{prop}

\begin{proof}
Note that $N=\cup \{xR  \mid x \in N\}$ as sets.  Note that if $N$ has no cyclic $\cl$-reductions then for any $x \in N$, $xR \in {\bf I}'_{\cl}(N)$ and there is a maximal element $\ca_{x}$ of ${\bf I}'_{\cl}(N)$ containing $xR$.  Since $$N=\cup \{xR \mid x \in N\} \subseteq  \cup \{\ca_x \mid x \in N\}=\cup\{\ca \mid \ca \text{ a } {\cl} \text{-prereduction of } N\} \subseteq N$$
we see that $N=\cup \{\ca \mid \ca \text{ a } \cl \text{-prereduction of } N\}$ if $N$ has no cyclic $\cl$-reductions in $M$.

Suppose now that $N$ has a cyclic $\cl$-reduction $xR$.  Then for any submodule $L$ of $M$ with $xR \subseteq L \subseteq N$, $L$ is a $\cl$-reduction of $N$.  Thus no $\cl$-prereduction of $N$ contains $xR$, thus $$\cup\{\ca \mid \ca \text{ a } {\cl} \text{-prereduction of } N\} \subsetneq N.$$
\end{proof}

Numerical semigroup rings gives a nice source of examples where we can easily exhibit $\cl$-prereductions for various closures.  

\begin{example} \label{ex:prered}
Let $R=k[[x^2,x^5]]$, $\cm=(x^2,x^5)$ and $k$ a field of any characteristic. We can find the $\cm$bf-closures for some of the non-zero non-unital ideals of $R$.

$(x^2,x^5)_R^{\cm bf}= (x^2,x^5)=\cm$.\\
For $n\neq 3$, $$(x^n)_R^{\cm bf}=((x^{n+2},x^{n+5}):_R (x^2,x^5))=(x^n,x^{n+3}).$$\\
For $n\geq 4$, $$(x^n,x^{n+1})_R^{\cm bf}=((x^{n+2},x^{n+3)}:_R (x^2,x^5))=(x^n,x^{n+1}).$$\\
For $n=2$ and $n\geq 4$, $$(x^n,x^{n+3})_R^{\cm bf}=((x^{n+2},x^{n+5}):_R (x^2,x^5))=(x^n,x^{n+3}).$$\\

Let $I_n=(x^n,x^{n+1})$ for $n \geq 4$. Then $I_n^{-}=(x^n,x^{n+1})=I$ and $I_n^{\mbf}=I_n$.
$(x^n,x^{n+3})$ is a $\mbf$-prereduction of $I_n$.
$(x^{n+1},x^{n+2})$ is both an integral prereduction of $I_n$ and a $\mbf$-prereduction of $I_n$.
\end{example}

\section{$\ri$-expansions and $\ri$-postexpansions}

The second author along with Epstein and R.G. defined the dual notions to $\cl$-reduction and $\cl$-core, $\ri$-expansion and $\ri$-hull in \cite{ERGV-nonres} which we state below.

\begin{defn}
\cite[Definition 2.13]{ERGV-nonres} Let $R$ be a Noetherian ring, $\cM$ be a class of $R$-modules and $\ri$ be an interior operation defined on $\cP$ of pairs of modules $(A,B)$ with $A \subseteq B$ in $\cM$.   Suppose $(A,B), (C,B) \in \cP$.
\begin{enumerate}
\item We say that $C$ is an \textit{$\ri$-expansion of $A$ in $B$}  if $C^B_{\ri}\subseteq A \subseteq C$. 
\item We say that a submodule $C \subseteq B$ is \textit{$\ri$-cobasic} if $C_{\ri}^B=C$.
\item If $C$ is an $\ri$-expansion of $A$ in $B$ and there is no submodule $D\subseteq B$ such that $D^B_{\ri}=A^B_{\ri}$ then we say that $C$ is a \textit{maximal} $\ri$-expansion of $A$ in $B$.
\item We define the \textit{$\ri$-hull} by ${\rihull{B}{A}}= \sum \{ C \mid C^B_{\ri} \subseteq A \subseteq C \text{ and } (C,B) \in \cP\}$.
\end{enumerate}
\end{defn}

\begin{defn}
Let $(R, \m )$ be a local ring and $\ri$ an interior operation on Artinian R-modules. We say that $\ri$ is a \textit{Nakayama interior} if for any Artinian R-modules $A \subseteq C \subseteq B$, if $(A :_C\m )^B_{\ri} \subseteq A$, then $A^B_{\ri}=C^B_{\ri}$.
\end{defn}

It is known that maximal $\ri$-expansions exist for a submodule $A$ of $B$ in the following cases: if $(R, \m)$ is a complete local ring, $\cM$ is the class of Artinian $R$-modules and $\cl$ is a Nakayama interior \cite[Proposition 6.4]{ERGV-corehull} or when $R$ is an associative ring and if there exists an $\ri$-expansion $C$ of $A$ such that $B/C$ is Noetherian \cite[Proposition 6.5]{ERGV-corehull}.  When maximal $\ri$-expansions of $A$ exist in $B$, then 
\[{\rihull{B}{A}}=\sum \{ C \mid C \text{ a maximal {\ri}-expansion of } A \text{ and } (C,B) \in \cP\}.\]

We will now switch our focus to interior operations $\ri$.  As Kemp, Ratliff and Shah defined the set of ideals ${\bf I}'(I)$ to be the ideals contained in $I$ which are not (integral) reductions of $I$, we can dually defined the set 
\[
{\bf C}_{\ri}'(A,B)=\{A \subseteq C \subseteq B \mid C \text{ not an {\ri}-expansion of } A  \text{ in } B\}.
\]

\begin{defn}
We say $C$ is an \textit{{\ri}-postexpansion of $A$ in $B$} if $A\subseteq C \subseteq B$, $C$ not an {\ri}-expansion of $A$ in $B$, and for all submodules $D$ such that $A\subseteq D \subseteq C \subseteq B$, $D$ is an {\ri}-expansion of $A$ in $B$.
\end{defn}

Note that the maximal elements of ${\bf I}_{\cl}^\prime(N,M)$ are {\cl}-prereductions and the minimal elements of ${\bf C}_{\ri}^\prime (A,B)$ are {\ri}-postexpansions.  The following properties hold for ${\bf C}_{\ri}'(A,B)$:

\begin{prop} \label{pr:nonexp}
 Let $(R,\cm)$ be an Noetherian local ring and $\cP$ be Artinian $R$-modules. Let {\ri} a Nakayama interior on $\cP$. Let $A\subseteq B$ be R-modules such that ${\bf C}_{\ri}^\prime (A,B) \neq \emptyset$. Then the following hold:
\begin{enumerate}
    \item \label{it:nonexp1} Suppose $D \in {\bf C}_{\ri}^\prime (A,B)$ and $(D,C)\in\cP$ submodules of $B$. Then $C \in {\bf C}_{\ri}^\prime (A,B)$.
    \item \label{it:nonexp2} Let $C \in {\bf C}_{\ri}^\prime (A,B)$. Then there exists a element $\cA \in {\bf C}_{\ri}^\prime(A,B)$ minimal in ${\bf C}_{\ri}^\prime (A,B)$ with $\cA \subseteq C$.
    \item \label{it:nonexp3} Suppose that $\cA_1$ and $\cA_1$ are both minimal submodules in ${\bf C}_{\ri}^\prime(A,B)$. Then $\cA_1\cap \cA_2 \notin {\bf C}_{\ri}^\prime(A,B)$ and $\cA_1\cap \cA_2$ is an {\ri}-expansion of $A$ in $B$.
    \item \label{it:nonexp4} If $C,D\in {\bf C}_{\ri}^\prime(A,B)$ then either $C\cap D \in {\bf C}_{\ri}^\prime(A,B)$ or $C\cap D$ is an {\ri}-expansion of $A$ in $B$.
    \item \label{it:nonexp5} If $C\in {\bf C}_{\ri}^\prime(A,B)$ then $(A:_C\cm)\in {\bf C}_{\ri}^\prime(A,B)$.
    \item \label{it:nonexp6}  If $C\in {\bf C}_{\ri}^\prime(A,B)$ then $C_{\ri}^B + A\in {\bf C}_{\ri}^\prime(A,B)$.
    \item \label{it:nonexp7} If $\cA$ is minimal in ${\bf C}_{\ri}^\prime(A,B)$ then $(\cA_{\ri}^B\cap(A:_C\cm)_{\ri}^B)_{\ri}^B+A=\cA$ for any $C \in {\bf C}_{\ri}^\prime(A,B)$.
\end{enumerate}
\end{prop}

\begin{proof}

    \eqref{it:nonexp1} If $D\in {\bf C}_{\ri}^\prime(A,B)$, then $D$ is not an {\ri}-expansion of $A$. Since $D\subseteq C$ and $D_B^{\ri} \subseteq C_B^{\ri} \subsetneq A_B^{\ri}$ then $D$ is also not an {\ri}-expansion of $A$ in $B$.
    
    \eqref{it:nonexp2} Since $C\in {\bf C}_{\ri}^\prime(A,B)$ and R Artinian, then there exists an element $\cA\in {\bf C}_{\ri}^\prime(A,B)$ which is minimal in ${\bf C}_{\ri}^\prime(A,B)$ and $\cA\subseteq C$.
    
    \eqref{it:nonexp3} Since $\cA_1$ and $\cA_2$ are both minimal in ${\bf C}_{\ri}^\prime(A,B)$ and $A\subseteq \cA_1\cap \cA_2\subseteq \cA_1,\cA_2 \subseteq B$, then $\cA_1\cap \cA_2$ is an {\ri}-expansion of $A$ in $B$.
    
    \eqref{it:nonexp4} Since $A \subseteq C,D\subseteq B$ then $A \subseteq C\cap D\subseteq B$.  If $C\cap D \in {\bf C}_{\ri}^\prime(A,B)$, we are done. If $C\cap D \notin {\bf C}_{\ri}^\prime(A,B)$, then by definition $C\cap D$ is an {\ri}-expansion of $A$ in $B$.
    
    \eqref{it:nonexp5} Suppose $(A:_C\cm)\notin {\bf C}_{\ri}^\prime(A,B)$. So $(A:_C\cm)$ is an {\ri}-expansion of $A$ in $B$. Then 
    \[(A:_C\cm)_{\ri}^B\subseteq A \subseteq (A:_C\cm).\] Because {\ri} is a Nakayama interior, $A_{\ri}^B= C_{\ri}^B$ and $C$ is an {\ri}-expansion of $A$ in $B$ which contradicts our assumption. So $(A:_C\cm)\in {\bf C}_{\ri}^\prime(A,B)$.
    
    \eqref{it:nonexp6} Suppose that $ C_{\ri}^B +A\notin {\bf C}_{\ri}^\prime(A,B)$. Then $ C_{\ri}^B+A$ is an {\ri}-expansion of $A$ in $B$ and $({C}_{\ri}^B+A)_{\ri}^B=A_{\ri}^B$. Then ${C}_{\ri}^B\subseteq ({C}_{\ri}^B+A)_{\ri}^B=A_{\ri}^B \subseteq {C}_{\ri}^B$. This is the case because $\ri$ is intensive and order preserving and ${C}_{\ri}^B\subseteq {C}_{\ri}^B+A$ and $A\subseteq C$. Hence, $C$ is an {\ri}-expansion of $A$ in $B$ which contradicts $C\in {\bf C}_{\ri}^\prime(A,B)$.
    
    \eqref{it:nonexp7} Since $(\cA \cap (A:_C\cm))_{\ri}^B \subseteq \cA_{\ri}^B\cap(A:_C\cm)_{\ri}^B \subseteq \cA \cap (A:_C\cm)$, we have \[((\cA \cap (A:_C\cm))_{\ri}^B)_{\ri}^B \subseteq (\cA_{\ri}^B\cap(A:_C\cm)_{\ri}^B)_{\ri}^B \subseteq (\cA \cap (A:_C\cm))_{\ri}^B.\] So $(\cA_{\ri}^B\cap(A:_C\cm)_{\ri}^B)_{\ri}^B = (\cA \cap (A:_C\cm))_{\ri}^B$. By (5), $(A:_C\cm) \in {\bf C}_{\ri}^\prime(A,B)$ and by (4) $\cA\cap(A:_C\cm)\in {\bf C}_{\ri}^\prime(A,B)$. Since $\cA \cap(A:_C\cm) \subseteq \cA$ and $\cA$ is minimal, $\cA\cap (A:_C\cm)=\cA$ and $(\cA_{\ri}^B\cap((A:_C\cm))_{\ri}^B)_{\ri}^B = (\cA \cap (A:_C\cm))_{\ri}^B=\cA_{\ri}^B$. By (6), $(\cA_{\ri}^B\cap((A:_C\cm))_{\ri}^B)_{\ri}^B+A\in {\bf C}_{\ri}^\prime(A,B)$. Hence, by the minimality of $\cA$, \[\cA=(\cA_{\ri}^B\cap((A:_C\cm))_{\ri}^B)_{\ri}^B+A.\]
\end{proof}

\begin{prop} \label{pr:postexp}
 Let $(R,\cm)$ be a Noetherian ring and $\cP$ be Artinian $R$-modules. Let $\ri$ be a Nakayama interior on $\cP$. Then
\begin{enumerate}
    \item \label{it:postexp1} Every submodule $C\subseteq B$ which is not an $\ri$-expansion of $A$ in $B$ contains an $\ri$-postexpansion of $A$ in $B$.
    \item If $\cA$ is an $\ri$-postexpansion of $A$ in $B$ then
    \begin{enumerate}
        \item \label{it:postexp2a} $(A:_B\cm)\supseteq \cA$.
        \item \label{it:postexp2b} $\cA$ is $\ri$-open in $B$ or $\cA_{\ri}^B+A=\cA$.
        \end{enumerate}
\end{enumerate}
\end{prop}

\begin{proof}

    \eqref{it:postexp1} By Proposition \ref{pr:nonexp}\eqref{it:nonexp2}, there is some minimal element $\cA$ of ${\bf C}_{\ri}^\prime(A,B)$ with $\cA\subseteq C$. Such an $\cA$ must be an $\ri$-postexpansion since any submodule it contains must be an $\ri$-expansion of $A$ in $B$.
    \item \begin{enumerate}
        \item In the proof of Proposition \ref{pr:nonexp}\eqref{it:nonexp7}, we saw $\cA=\cA\cap (A:_B \cm)$. This implies $(A:_B\cm)\supseteq \cA$.
        \item We also saw in the proof of Proposition \ref{pr:nonexp}\eqref{it:nonexp7} that $\cA_{\ri}^B +A = (\cA\cap (A:_B \cm))_{\ri}^B +A =\cA$.
    \end{enumerate}

\end{proof}

\begin{cor} \label{co:postexpop}
 Let $(R,\cm)$ be a Noetherian ring and $\cP$ be Artinian $R$-modules and $\ri$ a Nakayama interior on the submodules of $R$. For every $\ri$-postexpansion  $\cA$ of ${\bf C}_{\ri}^B$ of $A$ in $B$, $(A:_C \cm) \supseteq \cA$ and $\cA=\cA_{\ri}^B$.
\end{cor}

\begin{proof}
Note that $\cA$ is an $\ri$-postexpansion of ${\bf C}_{\ri}^B$ since $\cA\in {\bf C}_{\ri}^B$ and every submodule $D$ with $C\subseteq D \subseteq \cA$, $D$ is an $\ri$-postexpansion of $C$ and hence an $\ri$-expansion of ${\bf C}_{\ri}^B$. By Proposition \ref{pr:postexp}\eqref{it:postexp2a},  $(A:_C \cm)\supseteq \cA$ and by Proposition \ref{pr:postexp}\eqref{it:postexp2b}, $\cA_{\ri}^B=\cA_{\ri}^B+A=\cA$.
\end{proof}

\begin{prop} \label{pr:postexpcompare}
 Let $(R,\cm)$ be a Noetherian local ring and $\ri$ a Nakayama interior on the modules of $R$. Let $A\subseteq C\subseteq B$ be submodules of $R$ with $C$ an $\ri$-expansion of $A$ in $B$. Then
\begin{enumerate}
    \item \label{it:postexpcomp1} ${\bf C}_{\ri}^\prime(C,B)\subseteq {\bf C}_{\ri}^\prime(A,B)$.
    \item \label{it:postexpcomp2} For each $D\in {\bf C}_{\ri}^\prime(A,B)$, $D+C \in {\bf C}_{\ri}^\prime(C,B)$.
    \item\label{it:postexpcomp3}  For each minimal element $\cA$ of ${\bf C}_{\ri}^\prime(C,B)$ there exists a minimal element $\cB$ of ${\bf C}_{\ri}^\prime(A,B)$ such that $\cA+C= \cB$.
\end{enumerate}
\end{prop}

\begin{proof}

    \eqref{it:postexpcomp1} Let $D\in {\bf C}_{\ri}^\prime(C,B)$. Since $A\subseteq C$, then $A\subseteq C \subseteq D$ and $A_{\ri}^B \subseteq {\bf C}_{\ri}^B \subsetneq D_{\ri}^B$. Since $D$ is not an $\ri$-expansion of $C$, it cannot be an $\ri$-expansion of the smaller module $A$. Thus $C\in {\bf C}_{\ri}^\prime(A,B)$.
    
    \eqref{it:postexpcomp1} If $D\in {\bf C}_{\ri}^\prime(A,B)$, then $A\subseteq D$ and $D$ is not an $\ri$-expansion of $A$. Note that $D\subseteq D+C$. To see that $D+C\in {\bf C}_{\ri}^\prime(C,B)$, it is enough to see that $D+C$ is not an $\ri$-expansion of $C$. Suppose that $D+C$ is an $\ri$-expansion of $C$. Then $(D+C)_{\ri}^B={\bf C}_{\ri}^B$. Note that $(D+C)_{\ri}^B\subseteq D_{\ri}^B +{\bf C}_{\ri}^B$. Since $C$ is an $\ri$-expansion of $A$, $A_{\ri}^B={\bf C}_{\ri}^B\subseteq D_{\ri}^B \subseteq A_{\ri}^B$  which gives a contradition to $D\in {\bf C}_{\ri}^\prime(A,B)$. Hence $D+C \in {\bf C}_{\ri}^\prime(C,B)$.
    
    \eqref{it:postexpcomp1} Let $\cA$ be a minimal element of ${\bf C}_{\ri}^\prime(C,B)$. By (1), ${\bf C}_{\ri}^\prime(C,B) \subseteq {\bf C}_{\ri}^\prime(A,B)$. Thus $\cA\in {\bf C}_{\ri}^\prime(A,B)$ and there must be a minimal element $\cB\in {\bf C}_{\ri}^\prime(A,B)$ with $\cB\subseteq \cA$. By (2), $\cA+C\in {\bf C}_{\ri}^\prime(C,B)$. Since $\cB\subseteq\cA +C$ and $\cB$ is minimal, we get $\cA+C=\cB$.

\end{proof}

\begin{prop} \label{pr:postexpcontain}
 Let $(R,\cm)$ be a Noetherian local ring and $\ri$ a Nakayama interior on the modules of $R$. If $\cA$ is an $\ri$-postexpansion of $A$ in $B$ and $A\subseteq C \subseteq \cA$, then $\cA$ is an $\ri$-postexpansion of $C$ in $B$.
\end{prop}

\begin{proof}
Since $\cA$ is an $\ri$-postexpansion of $A$ in $B$ and $A\subseteq C\subseteq \cA$, then $C$ is an $\ri$-expansion of $A$ in $B$. Also, $\cA\in {\bf C}_{\ri}^\prime(A,B)$. By Proposition \ref{pr:postexpcompare}\eqref{it:postexpcomp2}, $\cA+C=\cA\in {\bf C}_{\ri}^\prime(A,B)$.\\
Suppose that $\cA$ is not an $\ri$-postexpansion of $C$ in $B$. Then there exists a minimal $\cB\in {\bf C}_{\ri}^\prime(C,B)$ with $\cB\subsetneq\cA$ and $\cB$ an $\ri$-postexpansion of $C$ in $B$. Then by Proposition \ref{pr:postexpcompare}\eqref{it:postexpcomp3} and since minimal elements of ${\bf C}_{\ri}^\prime(A,B)$ are $\ri$-postexpansions of $A$ in $B$, there exists an $\ri$-postexpansion $\mathfrak{C}$ of $A$ in $B$ such that $\mathfrak{C}\subseteq \mathfrak{C}+C=\cB\subsetneq\cA$. This contradicts the minimality of $\cA$ in ${\bf C}_{\ri}^\prime(A,B)$. Hence, $\cA$ is an $\ri$-postexpansion of $C$ in $B$.
\end{proof}

\begin{example} \label{ex:postexp}
Let $R=k[[x^2,x^5]]$, $\cm=(x^2,x^5)$ and $k$ a field of any characteristic. We can find the $\cm$be-interiors for some of the non-zero non-unital ideals of $R$.
$$(x^2,x^5)_{\cm be}^R = (x^2,x^5) = \cm$$
$$(x^n)_{\mbe}^R=(x^{n+2},x^{n+5})$$
$$(x^n,x^{n+1})_{\cm be}^R = \begin{cases}
(x^4,x^7) & \textrm{if } n=4\\
(x^6,x^7) & \textrm{if } n=5\\
(x^n,x^{n+1}) & \textrm{if } n\geq 6\\
\end{cases}$$
$$(x^n,x^{n+3})_{\cm be}^R= \begin{cases}
(x^4,x^7) & \textrm{if } n=4\\
(x^7,x^8) & \textrm{if } n=5\\
(x^n,x^{n+3}) & \textrm{if } n\geq 6\\
\end{cases}$$

Let $I=(x^4,x^7)$. Then $(x^4,x^5)$ is a $\mbe$-expansion of $I$ since $(x^4,x^5)_{\mbe}^R=I_{\mbe}^R=I$ and $I\subseteq (x^4,x^5)$.\\
Since $(x^4)_{\mbe}^R=(x^6,x^9) \subseteq I_{\mbe}^R=I$ and $I_{\mbe}^R\not\subseteq (x^4)$, $I$ is a $\m$be-postexpansion of $(x^4)$. 
\end{example}

\section{Comparing $\cl$-prereductions and $\ri$-postexpansions for different closures and interiors}

There are many different closure and interior operations so it would be beneficial to know when closure operations are comparable and when interior operations are comparable. Knowing when they are comparable allows us to compare $\cl$-reductions, $\cl$-prereductions, $\ri$-expansions, and $\ri$-postexpansions.

\begin{defn}
Let $p_1$ and $p_2$ be pair operations defined on $\cP$ a collection of pairs of modules in $R$. We say $p_1 \leq p_2$ if $p_1(N,M)\subseteq p_2(N,M)$ for all $(N,M)\in \cP$. We say $p_1$ and $p_2$ are \textit{comparable} if $p_1\leq p_2$ or $p_2\leq p_1$.
\end{defn}

\begin{rmk}
Let $\cl_1$ and $\cl_2$ be closure operations defined on the submodules of $R$. We say $\cl_1\leq \cl_2$ if $N_M^{\cl_1}\subseteq N_M^{\cl_2}$ for all $(N,M)\in\cP$. So $\cl_1$ and $\cl_2$ are \textit{comparable} if $\cl_1\leq \cl_2$ or $\cl_2\leq \cl_1$. 

Let $\ri_1$ and $\ri_2$ be interior operations defined on the submodules of $R$. We say $\ri_1\leq \ri_2$ if $A_{\ri_1}^B \subseteq A_{\ri_2}^B$ for all $(A,B)\in\mathcal{P}$. So $\ri_1$ and $\ri_2$ are \textit{comparable} if $\ri_1\leq \ri_2$ or $\ri_2 \leq \ri_1$.
\end{rmk}

We first explore relationships between ${\cl}_i$-prereductions for $i=1,2$, when $\cl_1 \leq \cl_2$.

\begin{prop} \label{pr:clintfacts}
Let $R$ be a Noetherian ring.
\begin{enumerate}
    \item Let $\cM$ be the category of finitely $R$-modules and $\cP$ the class of pairs $(N,M)$, with $N \subseteq M$ in $\cM$.  Suppose $\cl_1\leq\cl_2$  are Nakayama closure operations.  Then we have
    \begin{enumerate}
        \item \label{it:clint1a} $(N_M^{\cl_1})_M^{\cl_2}=N_M^{\cl_2}=(N_M^{\cl_2})_M^{\cl_1}$, and
        \item \label{it:clint1b} $N_M^{\cl_1}$ is a $\cl_2$-reduction of $N_M^{\cl_2}$.
    \end{enumerate}
    \item   Let $\cM$ be the category of Artinian $R$-modules and $\cP$ the class of pairs $(C,B)$, with $C \subseteq B$ in $\cM$.  Suppose $\ri_1\leq\ri_2$ are  Nakayama interior operations.  Then we have
    \begin{enumerate}
        \item \label{it:clint2a} $(A_{\ri_1}^B)_{\ri_2}^B=A_{\ri_1}^B=(A_{\ri_2}^B)_{\ri_1}^B$, and
        \item \label{it:clint2b} $A_{\ri_2}^B$ is an $\ri_1$-expansion of $A_{\ri_1}^B$.
    \end{enumerate}
\end{enumerate}
\end{prop}

\begin{proof}

    \eqref{it:clint1a} Since $N\subseteq N_M^{\cl_1}\subseteq N_M^{\cl_2}$, we have \[N_M^{\cl_2}\subseteq (N_M^{\cl_1})_M^{\cl_2}\subseteq (N_M^{\cl_2})_M^{\cl_2}=N_M^{\cl_2}\] which implies $(N_M^{\cl_1})_M^{\cl_2}=N_M^{\cl_2}$. Also, note that 
    \[N_M^{\cl_2}\subseteq (N_M^{\cl_2})_M^{\cl_1}\subseteq (N_M^{\cl_2})_M^{\cl_2}=N_M^{\cl_2}\] yields $N_M^{\cl_2}=(N_M^{\cl_2})_M^{\cl_1}$.
        
    \eqref{it:clint1b} By \eqref{it:clint1a}, we have $(N_M^{\cl_1})_M^{\cl_2}=N_M^{\cl_2}$ and by definition we get that $N_M^{\cl_1}$ is a $\cl_2$-reduction of $N_M^{\cl_2}$.

    \eqref{it:clint2a} Since $A_{\ri_1}^B \subseteq A_{\ri_2}^B \subseteq A$, we have 
    \[A_{\ri_1}^B = (A_{\ri_1}^B)_{\ri_1}^B \subseteq (A_{\ri_2}^B)_{\ri_1}^B \subseteq A_{\ri_1}^B\] 
    which implies $A_{\ri_1}^B=(A_{\ri_2}^B)_{\ri_1}^B$. Also, note that \[A_{\ri_1}^B=(A_{\ri_1}^B)_{\ri_1}^B \subseteq (A_{\ri_1}^B)_{\ri_2}^B \subseteq A_{\ri_1}^B\] yields $(A_{\ri_1}^B)_{\ri_2}^B=A_{\ri_1}^B$.
        
    \eqref{it:clint2b} By \eqref{it:clint2a}, we have $A_{\ri_1}^B=(A_{\ri_2}^B)_{\ri_1}^B$ and by definition we get that $A_{\ri_2}^B$ is an $\ri_1$-expansion of $A_{\ri_1}^B$ in $B$.
\end{proof}

\begin{prop} \label{pr:compclsets}
Let $(R,\cm)$ be a Noetherian local ring and $\cl_1\leq\cl_2$ be Nakayama closures on $\cP$.
\begin{enumerate}
    \item \label{it:compclsets1} ${\bf I}_{\cl_2}^\prime(N,M) \subseteq {\bf I}_{\cl_1}^\prime(N,M)$ for all $(N,M)\in\cP$.
    \item  \label{it:compclsets2} If ${\bf I}_{\cl_2}^\prime(N,M) \neq \emptyset$, then ${\bf I}_{\cl_1}^\prime(N,M) \neq\emptyset$.
    \item  \label{it:compclsets3} Suppose $K\in {\bf I}_{\cl_1}^\prime(N,M)$. Then $K\in {\bf I}_{\cl_2}^\prime(N,M)$ if and only if $K$ is not a $\cl_2$-reduction of $N$ in $M$.
    \item  \label{it:compclsets4} If $L\in {\bf I}_{\cl_1}^\prime(N,M)$, then $L_M^{\cl_1}\cap N \in {\bf I}_{\cl_2}^\prime(N,M)$ if and only if $L\in {\bf I}_{\cl_2}^\prime(N,M)$.
    \item  \label{it:compclsets5} If $L\in {\bf I}_{\cl_1}^\prime(N,M)$, then $L+\cm N \in {\bf I}_{\cl_2}^\prime(N,M)$ if and only if $L\in {\bf I}_{\cl_2}^\prime(N,M)$.
\end{enumerate}
\end{prop}

\begin{proof}

    \eqref{it:compclsets1} Suppose $L\in {\bf I}_{\cl_2}^\prime(N,M)$. Then $L$ is not a $\cl_2$-reduction of $N$ in $M$. So $L\subseteq N$ and $L_M^{\cl_2}\subsetneq N_M^{\cl_2}$. If $L$ were a $\cl_1$-reduction of $N$ in $M$, then $N_M^{\cl_1}=L_M^{\cl_1}\subseteq L_M^{\cl_2}$ which implies $N_M^{\cl_2}=(N_M^{\cl_1})_M^{\cl_2} \subseteq L_M^{\cl_2} \subsetneq N_M^{\cl_2}$ which is a contradiction. Thus $L$ is not a $\cl_1$-reduction of $N$ in $M$ and so $L\in {\bf I}_{\cl_1}^\prime(N,M)$.
    
    \eqref{it:compclsets2} Since ${\bf I}_{\cl_2}^\prime(N,M)\neq \emptyset$ and ${\bf I}_{\cl_2}^\prime(N,M)\subseteq {\bf I}_{\cl_1}^\prime(N,M)$, then ${\bf I}_{\cl_1}^\prime(N,M)\neq \emptyset$.
    
    \eqref{it:compclsets3} If $K\in {\bf I}_{\cl_1}^\prime(N,M)$, then $K_M^{\cl_1}\subsetneq N_M^{\cl_1}$. Since $K_M^{\cl_2}=(K_M^{\cl_1})_M^{\cl_2}\subseteq (N_M^{\cl_1})_M^{\cl_2}=N_M^{\cl_2}$, then $K\in {\bf I}_{\cl_2}^\prime(N,M)$ if and only if $K_M^{\cl_2}\neq N_M^{\cl_2}$.
    
    \eqref{it:compclsets4} Suppose $L\in {\bf I}_{\cl_2}^\prime(N,M)$, then since $L_M^{\cl_1}\subseteq L_M^{\cl_2}$ and by Proposition \ref{pr:nonreds}\eqref{it:6nonred} $L_M^{\cl_2}\cap N \in {\bf I}_{\cl_2}^\prime(N,M)$, we see that by Proposition \ref{pr:compclsets}\eqref{it:compclsets1}, $L_M^{\cl_1}\cap N\in {\bf I}_{\cl_2}^\prime(N,M)$. \\
    Suppose $L\in {\bf I}_{\cl_1}^\prime(N,M) \backslash {\bf I}_{\cl_2}^\prime(N,M)$. Then $L_M^{\cl_1}\subsetneq N_M^{\cl_1}\subseteq N_M^{\cl_2}$. Since $L_M^{\cl_2}=(L_M^{\cl_1})_M^{\cl_2}=N_M^{\cl_2}$, we see $L_M^{\cl_1}$ is a $\cl_2$-reduction of $N$ in $M$ and $L\subseteq L_M^{\cl_1}\cap N\subseteq L_M^{\cl_2}=N_M^{\cl_2}$. Applying $\cl_2$ to this chain, we see that $L_M^{\cl_2}\subseteq (L_M^{\cl_1}\cap N)_M^{\cl_2}\subseteq (L_M^{\cl_2})_M^{\cl_2}=L_M^{\cl_2}$. Thus $L_M^{\cl_1}\cap N$ is a $\cl_2$-reduction of $N$ in $M$ and $L_M^{\cl_1}\cap N\notin {\bf I}_{\cl_2}^\prime(N,M)$.
    
    \eqref{it:compclsets5} If $L\in {\bf I}_{\cl_2}^\prime(N,M)$, then $L+\cm N \in {\bf I}_{\cl_2}^\prime(N,M)$ by Proposition \ref{pr:nonreds}\eqref{it:5nonred}. Suppose that $L\notin {\bf I}_{\cl_2}^\prime(N,M)$. Then $L+\cm N$ is either a $\cl_2$-reduction of $N$ or $N_M^{\cl_2}=(L+\cm N)_M^{\cl_2}$. Since $\cl_2$ is a Nakayama closure, then $N_M^{\cl_2}=L_M^{\cl_2}$. Thus $L\notin {\bf I}_{\cl_2}^\prime(N,M)$.
\end{proof}

\begin{prop} \label{pr:comparepre}
 Let $(R,\cm)$ be a Noetherian local ring and $\cl_1\leq\cl_2$ Nakayama closures defined on $\cP$.
\begin{enumerate}
    \item \label{it:comparepre1} For every $\cl_2$-prereduction $\cA$ of $N$ in $M$, there exists a $\cl_1$-prereduction $\cB$ with $\cA\subseteq\cB$.
    \item \label{it:comparepre2} If $N=N_M^{\cl_2}$ and $\cA$ is a $\cl_2$-prereduction of $N$ in $M$, then $\cA=\cA_M^{\cl_1}=\cA_M^{\cl_2}$.
    \item \label{it:comparepre3} If $N=N_M^{\cl_1}$ and $\cA$ is both a $\cl_1$- and $\cl_2$- prereduction of $N$ in $M$, then $\cA=\cA_M^{\cl_1}=\cA_M^{cl_2}\cap N$.
    \item \label{it:comparepre4} Suppose $\cA$ is a $\cl_1$-prereduction of $N$ in $M$ and $\cA=\cA_M^{\cl_2}$. Then $\cA$ is a $\cl_2$-prereduction of $N$ in $M$.
\end{enumerate}
\end{prop}

\begin{proof}

    \eqref{it:comparepre1} Since $\cA$ is a $\cl_2$-prereduction of $N$ in $M$, then it is a maximal element of ${\bf I}_{\cl_2}^\prime(N,M)$. By Proposition \ref{pr:compclsets}\eqref{it:compclsets1}, $\cA\in {\bf I}_{\cl_1}^\prime(N,M)$. By Proposition \ref{pr:nonreds}\eqref{it:2nonred} there then exists a maximal element $\cB\in {\bf I}_{\cl_1}^\prime(N,M)$ with $\cA\subseteq\cB$.
    
    \eqref{it:comparepre2} By Corollary \ref{co:closedprered} we know that $\cA=\cA_M^{\cl_2}$. Since $\cA\subseteq\cA_M^{\cl_2}\subseteq\cA_M^{\cl_2}$, we can conclude that $\cA=\cA_M^{\cl_1}=\cA_M^{\cl_2}$.
    
    \eqref{it:comparepre3} By Corollary \ref{co:closedprered} we know that $\cA=\cA_M^{\cl_1}$. By Proposition \ref{pr:prered}\eqref{it:2bprered} $\cA=\cA_M^{\cl_2}\cap N$. Combining the equalities gives the result.
    
    \eqref{it:comparepre4} Suppose $\cA$ is not a $\cl_2$-prereduction of $N$ in $M$. Then there exists a $\cB\in {\bf I}_{\cl_2}^\prime(N,M)$ with $\cA\subsetneq\cB$. Since ${\bf I}_{\cl_2}^\prime(N,M)\subseteq {\bf I}_{\cl_1}^\prime(N,M)$, then $\cB\in {\bf I}_{\cl_1}^\prime(N,M)$. Since $\cA$ is a $\cl_1$-prereduction of $N$ in $M$ and $\cA\subsetneq\cB$, $\cB$ must be a $\cl_1$-reduction of $N$ in $M$ which contradicts $\cB\in {\bf I}_{\cl_1}^\prime(N,M)$. Thus $\cA$ must be a $\cl_2$-prereduction of $N$ in $M$.
    \end{proof}

Now we move on to comparisons of ${\ri}_j$-postexpansions, when $j=1,2$ and $\ri_1 \leq \ri_2$.

\begin{prop} \label{pr:comparerisets}
 Let $(R,\cm)$ be a Noetherian ring and $\cP$ be Artinian $R$-modules and $\ri_1\leq\ri_2$ be Nakayama interiors on $\cP$.
\begin{enumerate}
   \item \label{it:compareriset1} ${\bf C}_{\ri_1}^\prime(A,B) \subseteq {\bf C}_{\ri_2}^\prime(A,B)$ for all $(A,B)\in\cP$.
    \item \label{it:compareriset2} If ${\bf C}_{\ri_1}^\prime(A,B) \neq \emptyset$, then ${\bf C}_{\ri_2}^\prime(A,B) \neq\emptyset$.
    \item \label{it:compareriset3} Suppose $C\in {\bf C}_{\ri_2}^\prime(A,B)$. Then $C\in {\bf C}_{\ri_1}^\prime(A,B)$ if and only if $C$ is not an $\ri_1$-expansion of $A$ in $B$.
    \item \label{it:compareriset4}  If $C\in {\bf C}_{\ri_2}^\prime(A,B)$, then $C_{\ri_2}^B+A \in {\bf C}_{\ri_1}^\prime(A,B)$ if and only if $C\in {\bf C}_{\ri_1}^\prime(A,B)$.
    \item \label{it:compareriset5} If $C\in {\bf C}_{\ri_2}^\prime(A,B)$, then $(A:_C\cm) \in {\bf C}_{\ri_1}^\prime(A,B)$ if and only if $C\in {\bf C}_{\ri_1}^\prime(A,B)$.
\end{enumerate}
\end{prop}

\begin{proof}

    \eqref{it:compareriset1} Suppose $C\in {\bf C}_{\ri_1}^\prime(A,B)$. Then $C$ is not an $\ri_1$-expansion of $A$ in $B$. So $A\subseteq C\subseteq B$ and $A_{\ri_1}^B\subsetneq C_{\ri_1}^B$. If $C$ was an $\ri_2$-expansion of $A$ in $B$, then $A_{\ri_2}^B=C_{\ri_2}^B\supseteq C_{\ri_1}^B$ which implies $A_{\ri_1}^B=(A_{\ri_2}^B)_{\ri_1}^B\supseteq (C_{\ri_2}^B)_{\ri_1}^B \supsetneq (C_{\ri_1}^B)_{\ri_1}^B= C_{\ri_1}^B$ which is a contradiction. Thus $C$ is not an $\ri_2$-expansion of $A$ in $B$ and so $C\in {\bf C}_{\ri_2}^\prime(A,B)$.
    
    \eqref{it:compareriset2} Since ${\bf C}_{\ri_1}^\prime(A,B)\neq\emptyset$ and ${\bf C}_{\ri_1}^\prime(A,B)\subseteq {\bf C}_{\ri_2}^\prime(A,B)$, then ${\bf C}_{\ri_2}^\prime(A,B)\neq\emptyset$.
    
    \eqref{it:compareriset3} If $C\in {\bf C}_{\ri_2}^\prime(A,B)$, then $A_{\ri_2}^B\subsetneq C_{\ri_2}^B$. Since $A_{\ri_1}^B= (A_{\ri_2}^B)_{\ri_1}^B \subseteq (C_{\ri_2}^B)_{\ri_1}^B = C_{\ri_1}^B$, then $C\in {\bf C}_{\ri_1}^\prime(A,B)$ if and only if $A_{\ri_1}^B\neq C_{\ri_1}^B$.
    
    \eqref{it:compareriset4} If $C\in {\bf C}_{\ri_1}^\prime(A,B)$, then since $C_{\ri_1}^B\subseteq C_{\ri_2}^B$ and by Proposition \ref{pr:nonexp}\eqref{it:nonexp6}, $C_{\ri_1}^B+A\in {\bf C}_{\ri_1}^\prime(A,B)$. We see that by Proposition \ref{pr:nonexp}\eqref{it:nonexp1}, $C_{\ri_2}^B+A\in {\bf C}_{\ri_1}^\prime(A,B)$. Suppose $C\in {\bf C}_{\ri_1}^\prime(A,B)\backslash {\bf C}_{\ri_1}^\prime(A,B)$. Then $A_{\ri_1}^B \subseteq A_{\ri_2}^B \subsetneq C_{\ri_2}^B$. Since $A_{\ri_1}^B = (C_{\ri_2}^B)_{\ri_1}^B=C_{\ri_1}^B$, we see $C_{\ri_2}^B$ is an $\ri_1$-expansions of $A$ in $B$ and $A_{\ri_1}^B = C_{\ri_1}^B \subseteq C_{\ri_1}^B+A \subseteq C_{\ri_2}^B +A \subseteq C$. Applying $\ri_1$ to this chain, we see that $A_{\ri_1}^B \subseteq C_{\ri_1}^B \subseteq (C_{\ri_2}^B+A)_{\ri_1}^B\subseteq C_{\ri_1}^B$. Thus $C_{\ri_2}^B+A$ is an $\ri_1$-expansion of $A$ in $B$ and $C_{\ri_2}^B+A\notin {\bf C}_{\ri_1}^\prime(A,B)$.
    
    \eqref{it:compareriset5} If $C\in {\bf C}_{\ri_1}^\prime(A,B)$, then $(A:_C \cm)\in {\bf C}_{\ri_1}^\prime(A,B)$ by Proposition \ref{pr:nonexp}\eqref{it:nonexp5}.
    Suppose that $C\notin {\bf C}_{\ri_1}^\prime(A,B)$. Then $(A:_C \cm)$ is either an $\ri_1$-expansion of $A$ in $B$ or $A_{\ri_1}^B=(A:_C\cm)_{\ri_1}^B$. Since $\ri_1$ is a Nakayama interior, then $A_{\ri_1}^B=C_{\ri_1}^B$. Thus $C\notin {\bf C}_{\ri_1}^\prime(A,B)$.
\end{proof}

\begin{prop} \label{pr:comparepost}
 Let $(R,\cm)$ be a Noetherian ring and $\cP$ be Artinian $R$-modules and $\ri_1\leq\ri_2$ be Nakayama interiors on the submodules of $R$. Then 
\begin{enumerate}
    \item \label{it:comparepost1} For every $\ri_1$-postexpansion $\cA$ of $A$ in $B$, there exists an $\ri_2$-postexpansion $\cB$ with $\cB\subseteq\cA$.
    \item \label{it:comparepost2} If $A=A_{\ri_1}^B$ and $\cA$ is an $\ri_1$-postexpansion of $A$ in $B$, then $\cA=\cA_{\ri_1}^B=\cA_{\ri_2}^B$.
    \item  \label{it:comparepost3} If $A=A_{\ri_2}^B$ and $\cA$ is both an $\ri_1$- and $\ri_2$- postexpansion on $A$ in $B$, then $\cA=\cA_{\ri_2}^B=\cA_{\ri_1}^B+A$.
    \item \label{it:comparepost4} Suppose $\cA$ is an $\ri_2$-postexpansion of $A$ in $B$ and $\cA=\cA_{\ri_1}^B$. Then $\cA$ is an $\ri_1$-postexpanion of $A$ in $B$.
\end{enumerate}
\end{prop}

\begin{proof}

    \eqref{it:comparepost1} Since $\cA$ is an $\ri_1$-postexpansion of $A$ in $B$, then it is a minimal element of ${\bf C}_{\ri_1}^\prime(A,B)$. By Proposition \ref{pr:comparerisets}\eqref{it:compareriset1} $\cA\in {\bf C}_{\ri_2}^\prime(A,B)$. By Proposition \ref{pr:nonexp}\eqref{it:nonexp2} there then exists a minimal element $\cB\in {\bf C}_{\ri_2}^\prime(A,B)$ with $\cB\subseteq\cA$.
    
    \eqref{it:comparepost2} By Corollary \ref{co:postexpop}, we know that $\cA=\cA_{\ri_1}^B$. Since $\cA_{\ri_1}^B \subseteq \cA_{\ri_2}^B \subseteq \cA$, we can conclude that $\cA=\cA_{\ri_1}^B=\cA_{\ri_2}^B$.
    
    \eqref{it:comparepost3} By Corollary \ref{co:postexpop}, we know that $\cA=\cA_{\ri_2}^B$. By Proposition \ref{pr:postexp}\eqref{it:postexp2b}, $\cA=\cA_{\ri_2}^B+A$. Combining the equalities gives the result.
    
    \eqref{it:comparepost4} Suppose $\cA$ is not an $\ri_1$-postexpansion of $A$ in $B$. Then there exists a $\cB\in {\bf C}_{\ri_1}^\prime(A,B)$ with $\cB\subsetneq\cA$. Since ${\bf C}_{\ri_1}^\prime(A,B)\subseteq {\bf C}_{\ri_2}^\prime(A,B)$, then $\cB\in {\bf C}_{\ri_2}^\prime(A,B)$. Since $\cA$ is an $\ri_2$-postexpansion of $A$ in $B$ and $\cB\subsetneq\cA$, $\cB$ must be an $\ri_2$-expansion of $A$ in $B$ which contradicts $\cB\in {\bf C}_{\ri_2}^\prime(A,B)$. Thus $\cA$ must be an $\ri_1$-postexpansion of $A$ in $B$.
\end{proof}

We provide a few examples illustrating the above Propositions.  For the first example we use the fact that the $\m$-basically full closure of an ideal is always contained in its integral closure.  In the second example we use the fact that the tight closure of an ideal is always contained in its integral closure.

\begin{examplex}

Let $R=k[[x^2,x^5]]$, $\m =(x^2,x^5)$ and $k$ is a field of any characteristic. Consider the ideal $I=(x^6, x^7)$.  Note that $I^{-}=(x^6,x^7)$ and $I^{\m {\rm bf}}=(\m I :_R\m)=(x^6,x^7)$.  Note that $(x^6,x^9)$ is an integral reduction of $I$ but not a basically full reduction of $I$ since $(x^6,x^9)^{\m {\rm bf}}=(x^6,x^9)$.  Since $I/(x^6,x^9) \cong k$, then $I$ is a non basically full cover of $(x^6,x^9)$.  Now by Remark \ref{rmk:basic}, we see that $(x^6,x^9)$ is a $\cm$ basically full prereduction of $I$ which is not an integral prereduction of $I$.  

However, the ideal $(x^7,x^8)$ is neither an integral reduction nor a basically full reduction of $I$ and $I/(x^7,x^8) \cong k$ implies that $I$ is a non-integral cover and a non-basically full cover of $(x^7,x^8)$.  Again we use Remark \ref{rmk:basic} to see that $(x^7,x^8)$ is both an integral prereduction of $I$ and a $\cm$ basically full prereduction of $I$.

In fact, since the integrally closed ideals in $R$ have the form $(x^n)k[[x]] \cap R$, then $(x^7,x^8)$ is the unique integral prereduction of $I$.  Whereas, $I$ has multiple basically full prereductions.   

This example also nicely illustrates Proposition \ref{prop:unionprered}.  $I$ has principal integral reductions $(f)$ where $f=\sum\limits_{n=6}^{\infty} a_nx^n$ and $a_6\neq 0$, and $I \neq (x^7,x^8)$ which is the union of its integral prereductions.  Whereas $(f)$ is not a basically full reduction of $I$ since $(f)^{\m {\rm bf}}=(a_6x^6+a_7x^7,x^9) \neq I$.  So $I$ has no principal basically full reductions and $I$ is clearly seen to be the union of its basically full prereductions. \qed
\end{examplex}

\begin{examplex}

Let $R=k[[x,y,z]]/(x^2-y^3-z^6)$, $\m=(x,y,z)$ and $k$ is a field of characteristic $p>3$.  $R$ has test ideal $\m$.  Note that $\m^2$ is both integrally closed and tightly closed.  $(y^2,z^2)$ is a minimal integral reduction of $\m$ which is not a minimal $*$-reduction of $\m^2$.  This is because $R$ is a Gorenstein ring and in this case $(y^2,z^2)^*=(y^2,z^2):_R \m=(xyz, y^2,z^2)$.  A minimal $*$-reduction of $\m^2$ is $(y^2,yz,z^2)$.  Since \[(y^2,yz,z^2)^*=((y,z^2) \cap (y^2,z))^*=(y,z^2)^* \cap (y^2,z)^*=(xz,y,z^2) \cap (xy,y^2,z)=\m^2\] by \cite[Proposition 2.4]{Va-mfull}.  
So although $(y^2,z^2)$ is an integral reduction of $\m^2$, $(y^2,z^2) \in {\bf I}'_{*}(\m^2)$.  Note that $(y^2,z^2)$ is not a $*$-prereduction of $\m^2$ because $(y^2,z^2) \subsetneq (y^2,z^2,xy,xz)^*=(y^2,z^2,xy,xz)$ by \cite[Theorem 2.2]{Vr-chains}.  In fact, we will see in the next section that $(y^2,z^2,xy,xz)$ is a $*$-prereduction of $\m^2$.  An example of an integral prereduction of $\m^2$ is $J=(xy,xz,y^3,yz,z^2)$; this is the case since $\m^2/J \cong k$ and for all $f \in \m^2 \setminus J$, $J+(f)=\m^2$ implying that $J$ is an integral prereduction by Proposition \ref{prop:preredcov}.
\qed    
\end{examplex}

\section{Duality}

In this section, we extend the work of \cite{ERG-duality} and \cite{ERGV-nonres} to show duality between $\cl$-prereductions and $\ri$-postexpansions. This will also help when discussing precores and posthulls in Section 8.

We will use $^\vee$ to denote the Matlis duality operation, $\textrm{Hom}_R(\_,E)$. If $\cP$ is the class of Matlis-dualizable $R$-modules, then for all $M\in\cP$, $M^{\vee\vee}\cong M$. In this section, $R$ is a complete local ring with maximal ideal $\cm$, residue field $k$, and $E:=E_R(k)$ the injective hull.

\begin{defn}
\cite[Definition 3.1]{ERGV-nonres} Let $R$ be a complete local ring. Let $p$ be a pair operation on a class of pairs of Matlis-dualizable $R$-modules $\cP$. For any pair of $R$-modules $(A,B)\in \cP^{\vee}$, set $$\cP^{\vee}:=\{(A,B)\mid ((B/A)^{\vee},B^{\vee})\in\cP\},$$ and define the dual of $p$ by 
$$p^{\dual}(A,B):=\left(\frac{B^{\vee}}{p((B/A)^{\vee},B^{\vee})}\right)^{\vee}.$$
\end{defn}

\begin{lemma} \label{lem:duality}
\cite[Lemma 3.3]{ERGV-nonres} Let $R$ be a complete local ring and $p$ a pair operation on a class of pairs of $R$-modules $\cP$. For any $(A,B)\in\cP$,
$$\left(\frac{B}{p(A,B)}\right)^{\vee}=p^{\dual}((B/A)^{\vee},B^{\vee}).$$
In particular, if $\cl$ is a closure operation then $((B/A)^{\vee})_{\cl^{\dual}}^{B^{\vee}}=(B/A_B^{\cl})^{\vee}$, and if $\ri$ is an interior operation, then $((B/A)^{\vee})_{B^{\vee}}^{\ri^{\dual}}=(B/A_{\ri}^B)^{\vee}$.
\end{lemma}

\begin{thm} \label{thm:clrediexp}
\cite[Theorem 6.2]{ERGV-nonres} Let $R$ be a Noetherian complete local ring. Let $\ri$ be a relative interior operation on pairs $A\subseteq B$ of $R$-modules that are Noetherian or Artinian, and let $\cl:=\ri^{\dual}$ be its dual closure operation. There exists an order reversing one-to-one correspondence between the poset of $\ri$-expansions of $A$ in $B$ and the poset of $\cl$-reductions of $(B/A)^{\vee}$ in $B^{\vee}$. Under this correspondence, an $\ri$-expansion $C$ of $A$ in $B$ maps to $(B/C)^{\vee}$, a $\cl$-reduction of $(B/A)^{\vee}$ in $B^{\vee}$
\end{thm}

\begin{thm} \label{thm:clpreredipostexp}
Let $R$ be a Noetherian complete local ring. Let $\ri$ be a relative interior operation on pairs $A\subseteq B$ of $R$-modules that are Noetherian or Artinian, and let $\cl:=\ri^{\dual}$ be its dual closure operation. There exists a one-to-one correspondence between the set of $\ri$-postexpansions of $A$ in $B$ and the set of $\cl$-prereductions of $(B/A)^{\vee}$ in $B^{\vee}$. Under this correspondence, an $\ri$-postexpansion $C$ of $A$ in $B$ maps to $(B/C)^{\vee}$, a $\cl$-prereduction of $(B/A)^{\vee}$ in $B^{\vee}$.
\end{thm}

\begin{proof}
$C$ is an $\ri$-postexpansion of $A$ in $B$ if and only if $A\subseteq C \subseteq B$ and $A_{\ri}^B \subsetneq {\bf C}_{\ri}^B$ and for all submodules $D$ with $A\subseteq D \subseteq C$ we have $A_{\ri}^B=D_{\ri}^B$. 
First, $A\subseteq C \subseteq B$ if and only if $(B/C)^\vee \subseteq (B/A)^\vee \subseteq B^\vee$ by properties of Matlis duality.
Next, $A_{\ri}^B \subsetneq C_{\ri}^B$ occurs if and only if
$$\left(\frac{B^{\vee}}{((B/A)^{\vee})_{B^{\vee}}^{\cl}}\right)^{\vee}\subsetneq \left(\frac{B^{\vee}}{((B/C)^{\vee})_{B^{\vee}}^{\cl}}\right)^{\vee}.$$
Since the modules in question are Matlis-dualizable and $(B/C)^{\vee}\subseteq (B/A)^{\vee}$, this happens if and only if 
$$((B/C)^{\vee})_{B^{\vee}}^{\cl}\subsetneq ((B/A)^{\vee})_{B^{\vee}}^{\cl}.$$
If for all submodules $D$ with $A\subseteq D \subseteq C$ we have $A_{\ri}^B=D_{\ri}^B$ then by Theorem \ref{thm:clrediexp}, $D$ is an $\ri$-expansion of $A$ in $B$ and thus maps to $(B/D)^{\vee}$ a $\cl$-reduction of $(B/A)^{\vee}$ in $B^{\vee}$ and $(B/C)^{\vee} \subseteq (B/D)^{\vee} \subseteq (B/A)^{\vee}$.
Similarly if for all $(B/D)^\vee$ with $(B/C)^{\vee} \subseteq (B/D)^{\vee} \subseteq (B/A)^{\vee}$ we have $((B/C)^\vee)_{B^\vee}^{\cl}=((B/A)^\vee)_{B^\vee}^{\cl}$ then again by Theorem \ref{thm:clrediexp}, we have $A_{\ri}^B=D_{\ri}^B$.

Thus $C$ is an $\ri$-postexpansion of $A$ in $B$ if and only if $(B/C)^{\vee}$ is a $\cl$-prereduction of $(B/A)^{\vee}$ in $B^{\vee}$.  
\end{proof}

\begin{example}
Let $R=k[[x^2,x^5]]$ and $E=kx \oplus kx^3 \oplus \bigoplus_{i=1}^\infty kx^{-i}$ where the $R$-action on $E$ is given on the monomials of $R$ by $$x^j x^{-i}= 
\begin{cases}
0 & \text{if } j-i=0,2 \text{ or } j-i\geq 4 \\
x^{j-i} & \text{if } j-i=1,3 \text{ or } j-i\leq -1\\
\end{cases}$$
Let \begin{align*}
    M&=kx+kx^3+\bigoplus_{i=1}^n kx^{-i},\\
    N&=kx+kx^3+\bigoplus_{i=1}^n kx^{-i}+kx^{-(n+2)}, \text{ and}\\
    K&=kx+kx^3+\bigoplus_{i=1}^{n}kx^{-i}+kx^{-(n+2)}+kx^{-(n+4)}.\\
\end{align*}
Then $\textrm{Ann}_R(M)=(x^{n+4}, x^{n+5})$, $\textrm{Ann}_R(N)=(x^{n+4},x^{n+7})$, and $\textrm{Ann}_R(K)=(x^{n+4})$.
From Example \ref{ex:prered}, we can see that $\textrm{Ann}_R(N)$ is a reduction of $\textrm{Ann}_R(M)$ and $\textrm{Ann}_R(K)$ a prereduction of $\textrm{Ann}_R(N)$. Then both by duality and by Example \ref{ex:postexp}, $N$ is an expansion of $M$ in $E$ and $K$ is a postexpansion of $N$ in $E$.
\end{example}

\begin{thm} \label{thm:clsetiset}
Let $R$ be a Noetherian complete local ring. Let $\ri$ be a relative interior operation on pairs $A\subseteq B$ of $R$-modules that are Noetherian or Artinian, and let $\cl:=\ri^{\dual}$ be its dual closure operation. There exists an order reversing one-to-one correspondence between between the elements of ${\bf I}_{\cl}^\prime((B/A)^\vee,B^\vee)$ and the elements of ${\bf C}_{\ri}^\prime(A,B)$.
\end{thm}

\begin{proof}
Let $C\in {\bf C}_{\ri}^\prime(A,B)$. If $C$ is an $\ri$-postexpansion of $A$ in $B$ then by Theorem \ref{thm:clpreredipostexp}, we are done.
If $C$ is not an $\ri$-postexpansion of $A$ in $B$ then $A\subseteq C \subseteq B$ and by Proposition \ref{pr:postexp}\eqref{it:postexp1}, $C$ contains an $\ri$-postexpansion of $A$ in $B$. Let that $\ri$-postexpansion be $D$. Then by the previous theorem, $D$ maps one-to-one to $(B/D)^\vee$ a $\cl$-prereduction of $(B/A)^\vee$ in $B^\vee$. Since $(B/C)^\vee \subseteq (B/D)^\vee$ and $(B/C)^\vee$ is not a $\cl$-reduction of $(B/A)^\vee$ in $B^\vee$ (otherwise it would map to $C$ and $C$ would be an $\ri$-expansion of $A$ in $B$), we get that $(B/C)^\vee \in {\bf I}_{\cl}^\prime ((B/A)^\vee, B^\vee)$.
The correspondence is order reversing since $$C\subseteq D \textrm{ if and only if } (B/D)^{\vee} \subseteq (B/C)^{\vee}.$$
\end{proof}

The following lemma will be useful when proving results that use generating or cogenerating sets.

\begin{lemma} \label{lem:sumintdual}
\cite[Lemma 6.15]{ERGV-corehull} Let $R$ be a complete Noetherian ring. Let $B$ be an $R$-module such that it and all of its quotient modules are Matlis-dualizable. Let $\{C_i\}_{i\in I}$ a collection of submodules of $B$. Then 
$$\left(\frac{B}{\sum _i C_i}\right)^{\vee} \cong \bigcap_i (B/C_i)^{\vee}$$ and
$$\left( \frac{B}{\bigcap _i C_i}\right)^{\vee} \cong \sum_i (B/C_i)^{\vee}$$
where all the dualized modules are considered as submodules of $B^{\vee}$.
\end{lemma}

\section{Covers: $\cl$-prereductions and $\ri$-postexpansions}

In this section we discuss the relationship between covers of submodules with respect to closure operations $\cl$ and interior operations $\ri$.

\begin{defn}
 Let $(R,\cm)$ be a local Noetherian ring and let $K, N$ be submodules of $M$ in $R$. We say that $N$ \textit{covers} $K$ if $K\subsetneq N$ and $N/K\cong R/\cm$. We also say that $K$ is covered by $N$.
 \end{defn}

\begin{defn}
 Let $(R,\cm)$ be a local Noetherian ring, $\cM$  be a class of $R$-modules, and $\cP$ a class of pairs of modules in $\cM$. 
 
 Suppose $\cl$ is a closure operation on pairs of modules $(L,M), (N,M)$ in $\cP$. If $L$ is covered by $N$:
\begin{enumerate}
    \item We say that $N$ is a \textit{$\cl$-cover} of $L$ if $N_M^{\cl}=L_M^{\cl}$.
    \item We say that $N$ is a \textit{non-$\cl$-cover} of $L$ if $L_M^{\cl}\subsetneq N_M^{\cl}$.
\end{enumerate}

 Let $\ri$ be an interior operation on pairs of modules $(A,B), (C,B)$ in $\cP$. If $A$ is covered by $C$:
\begin{enumerate}
    \item We say that $C$ is a \textit{$\ri$-cover} of $A$ if $A_{\ri}^B={ C}_{\ri}^B$.
    \item We say that $C$ is a \textit{non-$\ri$-cover} of $A$ if $A^B_{\ri}\subsetneq C^B_{\ri}$.
\end{enumerate}
\end{defn}

\begin{rmkx}\label{rmk:sumint}  $N$ covers $K$ if and only if $N=K+xR$ for some $x\in N$ and $\cm x \subseteq K$. Equivalently, $N$ covers $K$ if and only if $N=K+xR$ and $\cm=(K:_R xR)$ for every $x\in N\backslash K$. If $N$ covers $K$ and $L$ is an arbitrary proper submodule then either $N\cap L$ covers $K\cap L$ and $N+L=K+L$ or $N\cap L= K\cap L$ and $N+L$ covers $K+L$.
\end{rmkx}

\begin{rmkx}
Let $(R, \m)$ be a Noetherian local ring, $\cM$ be the category of finitely generated $R$-modules and $\cP$ be the set of pairs $(L,M)$ with $L \subseteq M$ and $L,M \in \cM$.  Suppose $L \subseteq N$ are submodules of $M$ with $(L,M), (N,M) \in \cP$.  The following are equivalent:

\begin{enumerate}
\item $L+xR$ is a non-$\cl$-cover of $L$ for all $x \in N \setminus L$.
\item $\m N \subseteq L$ and $L_M^{\cl} \cap N =K$.
\end{enumerate}
\end{rmkx}

\begin{proof}
This follows directly from definiton of non-$\cl$ cover and Remark \ref{rmk:sumint}.
\end{proof}

The following proposition is a generalization of \cite[Theorem 4.5]{KRS-prered} for Nakayama closures $\cl$.

\begin{prop} \label{prop:preredcov}
Let $(R, \m)$ be a Noetherian local ring, $\cM$ be the class of finitely generated $R$-modules, $\cP$ be the class of pairs $(L,M)$ with $L \subseteq M$ and $L,M \in \cM$ and $\cl$  be a Nakayama closure on $\cP$.  Let $\cA$ be a $\cl$-prereduction of $L$ in $M$.  For every $x \in L \setminus \cA$:

\begin{enumerate}
\item \label{it:preredcov1} $\cA+xR$ is a non-$\cl$-cover of $\cA$.
\item \label{it:preredcov2} $\cA$ is a $\cl$-prereduction of $\cA+xR$.
\item \label{it:preredcov3} $\cA + xR$ is a $\cl$-reduction of $L$.
\end{enumerate}
\end{prop}

\begin{proof}
\eqref{it:preredcov1}  It follows from Proposition \ref{pr:prered}\eqref{it:2aprered} that $\m L \subseteq \cA$ and from Remark \ref{rmk:sumint} that $\cA+xR$ is a cover of $\cA$.  Also $\cA_M^{\cl} \cap L=\cA$ follows from Proposition \ref{pr:prered}\eqref{it:2bprered}, so $x \notin \cA_M^{\cl}$.  Thus $\cA+xR$ is a non-$\cl$-cover of $\cA$.

\eqref{it:preredcov2} $\cA$ is a $\cl$-prereduction of $\cA+xR$ by Proposition \ref{pr:preredcontainment} since $ \cA \subseteq \cA+xR \subseteq L$.

\eqref{it:preredcov3} $\cA+xR$ is a $\cl$-reduction of $L$ in $M$ by the definition of $\cl$-prereduction of $L$.
\end{proof}

\begin{rmkx} \label{rmk:basic}
Let $(R,\cm)$ be a Noetherian local ring.
\begin{enumerate}
    \item \label{it:clbasic2} If $L$ is $\cl$-basic ($L$ is the only $\cl$-reduction inside itself), then $L$ is a non-$\cl$-cover of each $\cl$-prereduction of itself.
    \item  \label{it:clbasic3} If $K$ and $L$ are submodules of $M$ and $L$ is a non-$\cl$-cover of $K$, then $K$ is a $\cl$-prereduction of $L$.
\end{enumerate}
\end{rmkx}

\begin{proof}
\eqref{it:clbasic2} Since $L$ is $\cl$-basic and $\cA$ is a $\cl$-prereduction of $L$, then Proposition \ref{prop:preredcov}\eqref{it:preredcov3} shows that for all $x \in L \setminus \cA$, $\cA+xR$ is a reduction of $L$.  But this implies that $\cA+xR=L$ for all $x \in L\setminus \cA$.

\eqref{it:clbasic3} If $I$ is a non-$\cl$-cover of $J$, then $J$ is not a $\cl$-reduction of $I$.  Also by Remark \ref{rmk:sumint},  $J+(x)=I$ for all $x \in I \setminus J$.  Since $I$ is a reduction of $I$, then $J$ is a $\cl$-prereduction of $I$. 
\end{proof}

\begin{prop} \label{pr:clindep}
 Let $(R,\cm)$ be a Noetherian local ring and $N$ be a strongly $\cl$-independent submodule in $M$ with $\cl$-spread equal to $k\geq 1$ elements. Then every $\cl$-prereduction of $N$ in $M$ has the form $(y_1,y_2,...,y_{k-1})+y_k\cm$ where $y_1,...,y_k$ are a minimal generating set for $N$.
\end{prop}

\begin{proof}
Let $y_1,...,y_k$ be a minimal generating set for $N$ and let $\cA=(y_1,...,y_{k-1})+y_k\cm$. Then $\cA+(y_k)=(y_1,...,y_k)=N$ and $y_k \cm \subseteq \cA$, so $N$ is a cover of $\cA$ by Remark \ref{rmk:sumint}. Also, the $y_i$ are strongly $\cl$-independent, so $N$ is $\cl$-basic and $N$ is the only $\cl$-reduction of itself. Thus $N$ is a non-$\cl$-cover of $\cA$. hence, Remark \ref{rmk:basic}\eqref{it:clbasic3} implies that $\cA$ is a $\cl$-prereduction of $N$ in $M$.

Suppose $\cA$ is an arbitrary $\cl$-prereduction of $N=(x_1,...,x_k)$. Then there exists an $1\leq i \leq k$ such that $x_i\notin \cA$. Let us assume $i=k$, then $x_k \notin \cA$. Note that since $\cA$ is a $\cl$-prereduction of $N$ in $M$ then for any $x\in N\backslash\cA$, $\cA+(x)$ is a $\cl$-reduction of $N$ in $M$. However, since $N$ is $\cl$-basic, this implies that $N=\cA+(x)$ for any $x\in N\backslash \cA$. In particular, $N=\cA + (x_k)$. Thus for $1\leq i \leq k-1$, there exists $a_i\in \cA$ and $b_i\in R$ such that $x_i=a_i+b_ix_k$ and 
$$\{a_1+b_1x_k,...,a_{k-1}+b_{k-1}x_k,x_k\}$$
is a minimal generating set of $N$. Thus
$$\{a_1,...,a_{k-1},x_k\}$$
is also a minimal generating set of $N$. Since $\cm N \subseteq \cA$ by Proposition \ref{pr:prered}\eqref{it:2aprered}, we have
$$(a_1,...,a_{k-1})+x_k\cm \subseteq (a_1,...,a_{k-1})+\cm N \subseteq \cA.$$
Since $\cA$ is a $\cl$-prereduction of $N$ in $M$ and $a_1,...,a_{k-1})+x_k\cm$ is a $\cl$-prereduction of $N$ in $M$, we see that $\cA=(a_1,...,a_{k-1})+x_k\cm$.
\end{proof}

\begin{examplex}
Let $R=k[[x^2,x^5]]$ and $I=(x^6,x^7)$.  By definition of $\m$bf-independent, $x^6$ and $x^7$ are $\m$bf-independent since $x^6 \notin (x^7)^{\m{\rm bf}}=(x^7,x^{10})$ and  $x^7 \notin (x^6)^{\m{\rm bf}}=(x^6,x^{9})$.  Note that $(x^6)+\m(x^7)=(x^6,x^9)$ is an $\m$bf-prereduction and $(x^7)+\m(x^6)=(x^7,x^8)$ is an $\m$bf-prereduction. 
\end{examplex}

\begin{defn}
\cite[Definition 6.6]{ERGV-corehull} Let $R$ be a Noetherian local ring, $L$ an $R$-module, and $g_1,...,g_t\in L^\vee$. We say that the \textit{quotient of $L$ cogenerated by $g_1,...g_t$} is $L/(\bigcap_i \ker(g_i))$. 
We say that $L$ is \textit{cogenerated by $g_1,...g_t$} if $\bigcap_i\ker(g_i)=0$.
We say that a cogenerating set for $L$ is \textit{minimal} if it is irredundant, i.e., for all $1\leq j \leq t$, $\bigcap_{i\neq j} \ker(g_i)\neq 0$.
\end{defn}

We can dualize the notion of strongly $\cl$-independent generating set to that of a strongly $\ri$-independent cogenerating set as follows:

\begin{defn}
Let $R$ be a Noetherian local ring, $L \subseteq M$ $R$-modules, $\pi:M \rightarrow M/L$ the canonical projection, and $\ri$ an interior operation defined on $R$-modules. We say that $g_1,\ldots, g_k \in (M/L)^\vee$ are an \textit{$\ri$-independent cogenerating set} of $M/L$ if $\pi^{-1}({\rm ker}(g_i)) \not\supseteq (\pi^{-1}(\bigcap\limits_{r \neq i}{\rm ker} (g_r)))^M_{\ri}$ for any $1 \leq i \leq k$.  We say that $L$ is \textit{strongly $\ri$-independent} if any minimal set of cogenerators of $M/L$ is $\ri$-independent.
\end{defn}

\begin{example}
    Let $R=k[[x^2,x^5]]$.  Note that $(x^{11},x^{12})=(x^6) \cap (x^7)$ and $(x^{11},x^{14})=(x^6) \cap (x^9)$.  Define $g_i: R \rightarrow E$ to be the homomorphism defined by $g_i(1)=x^{-i}$ for every $i$ in the semigroup $\langle 2, 5\rangle$. Note that $g_i$ has \[{\rm ker}(g_i)=(x^i) \text{ and }
    {\rm im} (g_i)=kx^3+kx+\sum\limits_{j=1}^{i-4}x^{-j}+kx^{-i+2}+kx^{-i}\] for $i > 4$.  Since \[(R/(x^{11},x^{12}))^\vee \cong kx^3+kx+\sum\limits_{j=1}^{7}x^{-j} \text{ and } (R/(x^{11},x^{14}))^\vee \cong kx^3+kx+\sum\limits_{j=1}^{7}x^{-j}+kx^{-9},\] it is easy to see that  $R/(x^{11},x^{12})$ is cogenerated by the functions $g_6$ and $g_7$ whereas $R/(x^{11},x^{14})$ is cogenerated by $g_6$ and $g_9$.  
     
     Note that ${\rm ker}(g_6) \cap {\rm ker}(g_9)= (x^{11},x^{14})$ and ${\rm ker}(g_9)=(x^9)$ is a $\m$be expansion of $(x^{11},x^{14})$. Thus ${\rm ker}(g_6) \supseteq {\rm ker}(g_9)_{\m{\rm be}}$ implying that $g_6$ and $g_9$ are not strongly $\m$be-independent.  
     
     However, $g_6$ and $g_7$ will be strongly $\m$be-independent since ${\rm ker}(g_6) \not\supseteq (x^9, x^{12})=({\rm ker}(g_7))_{\m{\rm be}}$. 
\end{example}

\begin{prop} \label{prop:postexpcov}
Let $(R, \m)$ be a complete Noetherian local ring, $\cM$ be the class of Artinian $R$-modules, $\cP$ be the class of pairs $(A,B)$ with $A \subseteq B$ and $A,B \in \cM$, $\pi:B \rightarrow B/A$ the canonical surjection and $\ri$  be a Nakayama interior on $\cP$.  Let $\cA$ be an $\ri$-postexpansion of $A$ in $B$.  For every $g \in (B/A)^\vee$ such that $\pi^{-1}({\rm ker}(g)) \not\supseteq \cA$:

\begin{enumerate}
\item \label{it:postexpcov1} $\cA$ is a non-$\ri$-cover of $\cA \cap \pi^{-1}({\rm ker}(g))$.
\item \label{it:postexpcov2} $\cA$ is an $\ri$-postexpansion of $\cA\cap \pi^{-1}({\rm ker}(g))$.
\item \label{it:postexpcov3} $\cA \cap \pi^{-1}({\rm ker}(g))$ is an $\ri$-expansion of $A$.
\end{enumerate}
\end{prop}

\begin{proof}
\eqref{it:postexpcov1}  We need to show that $(\cA \cap \pi^{-1}({\rm ker}(g)))_{\ri}^B \subsetneq \cA_{\ri}^B$.  By Proposition \ref{pr:postexp}\eqref{it:postexp2b}, since $\cA$ is an $\ri$-postexpansion of $A$ in $B$ either $\cA$ is $\ri$-open in $B$ or $\cA_{\ri}^B+A=\cA$. In the first case, we have $\cA_{\ri}^B=\cA \supsetneq \cA \cap \pi^{-1}({\rm ker}(g)) \supseteq (\cA \cap \pi^{-1}({\rm ker}(g)))_{\ri}^B$. So $\cA$ is a non-$\ri$-cover of $\cA \cap \pi^{-1}({\rm ker}(g))$. In the other case, we have $\cA_{\ri}^B+A=\cA$. Then again we get \[(\cA \cap \pi^{-1}({\rm ker}(g)))_{\ri}^B = ((\cA_{\ri}^B+A) \cap \pi^{-1}({\rm ker}(g)))_{\ri}^B = ((\cA_{\ri}^B \cap \pi^{-1}({\rm ker}(g)))+(A\cap \pi^{-1}({\rm ker}(g))))_{\ri}^B \subsetneq \cA_{\ri}^B.\]

\eqref{it:postexpcov2} $\cA$ is an $\ri$-postexpansion of $\cA \cap \pi^{-1}({\rm ker}(g))$ by Proposition \ref{pr:postexpcontain} since $ A \subseteq \cA \cap \pi^{-1}({\rm ker}(g)) \subseteq \cA$.

\eqref{it:postexpcov3} $\cA \cap \pi^{-1}({\rm ker}(g))$ is an $\ri$-expansion of $A$ in $B$ by the definition of $\ri$-postexpansion of $A$.
\end{proof}

\begin{prop}
\cite[Proposition 6.14]{ERGV-corehull} Let $(R,\cm)$ be a Noetherian local ring and $\ri$ a Nakayama interior on Artinian $R$-modules. Let $A\subseteq B$ Artinian $R$-modules. Suppose that $C\subseteq D$ are $\ri$-expansions of $A$ in $B$, with $D$ a maximal $\ri$-expansion. Then any minimal cogenerating set of $B/D$ extends to a minimal cogenerating set for $B/C$.
\end{prop}

\begin{defn}
\cite[Definition 7.18]{ERGV-corehull} Let $(R,\cm)$ be a Noetherian local ring. Let $\ri$ be an interior operation defined on a class of Artinian $R$-modules $\cM$. Let $A\subseteq B$ be Artinian $R$-modules. We define the \textit{$\ri$-cospread} $\ell_{\ri}^B(A)$ of $A$ to be the minimal number of cogenerators of $B/C$ of any maximal $\ri$-expansion $C$ of $A$, if this number exists.
\end{defn}

\begin{prop}
Let $(R,\cm)$ be a Noetherian local ring, $\cl$ be a Nakayama closure operation on $R$-modules, and $\ri$ the interior operation dual to $\cl$ defined on a class of Artinian $R$-modules. Let $N\subseteq M$ and $L$ a $\cl$-prereduction of $N$ in $M$. If the $\cl$-spread $\ell_M^{\cl}(L)$ of $L$ in $M$ exists, then the $\cl$-spread $\ell_M^{\cl}(N)$ exists and $$\ell_M^{\cl}(N)=\ell_M^{\cl}(L)+1.$$
Let $A\subseteq B$ be Artinian $R$-modules, $C$ an $\ri$-postexpansion of $A$ in $B$, $M=B^{\vee}$, $N=(B/A)^{\vee}$, and $L=(B/C)^{\vee}$. Then the $\ri$-cospread $\ell_{\ri}^B(A)$ exists, the $\ri$-cospread $\ell_{\ri}^B(C)$ exists, and $$\ell_{\ri}^B(C)=\ell_{\ri}^B(A)+1.$$
\end{prop}

\begin{proof}
Suppose the $\cl$-spread $\ell_M^{\cl}(L)$ of $L$ in $M$ exists and $L$ a $\cl$-prereduction of $N$ in $M$. So for all $K$ such that $L\subseteq K \subseteq N$, $K$ is a $\cl$-reduction of $N$ in $M$. Let $J$ be a minimal $\cl$-reduction of $L$ in $M$ and $\tilde{K}$ be a minimal $\cl$-reduction of $N$ in $M$. Then $J \subseteq L \subseteq J_M^{\cl} \subseteq \tilde{K} \subseteq N \subseteq \tilde{K}_M^{\cl}$. Since $\tilde{K}$ is a minimal $\cl$-reduction of $N$ and $J$ is a minimal $\cl$-reduction of a $\cl$-prereduction of $N$, $J_M^{\cl}\subsetneq \tilde{K}_M^{\cl}$ and the minimal number of generators of $J$ is exactly 1 less than the minimal number of generators of $\tilde{K}$. Since minimal reductions of $L$ all have the same number of minimal generators and $\tilde{K}$ was an arbitrary minimal $\cl$-reduction of $N$ in $M$,
$$\ell_M^{\cl}(N)=\ell_M^{\cl}(L)+1.$$
By \cite[Proposition 7.19]{ERGV-corehull}, we know that since $\cl$-spread $\ell_M^{\cl}(L)$ and $\ell_M^{\cl}(N)$ exists, then the $\ri$-cospreads $\ell_{\ri}^B(A)$ and $\ell_{ri}^B(C)$ exists. Because $C$ is an $\ri$-postexpansion of $A$, we have
$\ell_{\ri}^B(C)=\ell_{\ri}^B(A)+1.$
\end{proof}

\begin{prop} \label{prop:cover}
Let $(R,\cm)$ be a complete Noetherian local ring.
\begin{enumerate}
    \item \label{it:cover1} 
    Suppose $C \subseteq A \subseteq B$ and $\pi: B \rightarrow B/C$ is the canonical surjection. $(B/C)^\vee$ covers $(B/A)^\vee$ if and only if $C=A\cap \pi^{-1}({\rm ker}(g))$ for some $g\in (B/C)^\vee$ and $(\pi^{-1}({\rm ker} (g)):_{B} \cm)\supseteq A$.  
    \item \label{it:cover2} If $C$ is $\ri$-cobasic then every $\ri$-postexpansion $\cA$ of $C$ is a non-$\ri$-cover of $C$. 
    \item \label{it:cover3} If $A$ and $C$ are submodules of $B$ and $C$ is a non-$\ri$-cover of $A$, then $C$ is an $\ri$-postexpansion of $A$.
    \end{enumerate}
\end{prop}

\begin{proof} \eqref{it:cover1}  Suppose $(B/C)^\vee$ covers $(B/A)^\vee$. Then by Remark \ref{rmk:sumint}, $(B/C)^\vee=(B/A)^\vee+(g)$ for some $g\in (B/C)^\vee$ and $\cm g \subseteq (B/A)^\vee$. 
By \cite[Lemma 5.4]{ERGV-nonres} $(g) \cong ((B/C)/({\rm ker}(g)))^\vee$.  By the third isomorphism theorem $(B/C)/({\rm ker}(g)) \cong B/(\pi^{-1}({\rm ker}(g))$. 
Then by Lemma \ref{lem:sumintdual} we see that 
\[(B/A)^\vee+(B/\pi^{-1}({\rm ker}(g)))^\vee \cong (B/(A\cap \pi^{-1}({\rm ker}(g))))^\vee\] or $(B/C)^\vee \cong (B/(A\cap \pi^{-1}({\rm ker}(g))))^\vee$. Thus $C=A\cap \pi^{-1}({\rm ker}(g))$ for some $g\in (B/C)^\vee$. Furthermore, since $\cm g \subseteq (B/A)^\vee$ then 
\[A= \left(\displaystyle\frac{(B/C)^\vee}{ (B/A)^\vee}\right)^\vee \subseteq \left(\displaystyle\frac{(B/C)^\vee}{\cm g}\right)^\vee=\left(\displaystyle\frac{B^\vee}{\cm (g \circ \pi^\vee)}\right)^\vee  = (\pi^{-1}({\rm ker}(g)):_{B} \cm).\] 
     
Suppose $C=A\cap \pi^{-1}({\rm ker}(g))$ for some $g\in (B/C)^\vee$ and $\pi^{-1}({\rm ker} g:_{B/C} \cm)\supseteq A$. Then by Lemma \ref{lem:sumintdual} $$(B/C)^\vee = \left(\frac{B}{A\cap \pi^{-1}({\rm ker}(g)}\right)^\vee \cong (B/A)^\vee+(B/\pi^{-1}({\rm ker}(g)))^\vee$$ and since $(g) \cong ((B/C)/({\rm ker}(g)))^\vee \cong (B/(\pi^{-1}({\rm ker}(g)))^\vee$, we see that $(B/C)^\vee = (B/A)^\vee +(g)$.
    
Since $(\pi^{-1}({\rm ker} (g)):_{B} \cm)\supseteq A$, we have 
\[(B/A)^\vee \supseteq \left(\displaystyle\frac{B}{(\pi^{-1}({\rm ker} (g)):_{B} \cm)} \right) =\left(\displaystyle\frac{B/C}{({\rm ker} (g)):_{B/C} \cm)}\right)^\vee=\m g\]
or $\cm g \subseteq (B/A)^\vee$ and by Remark \ref{rmk:sumint} $(B/C)^\vee$ covers $(B/A)^\vee$.

   \eqref{it:cover2} 
    Let $\cA$ be an $\ri$-postexpansion of $C$. Since $C$ is $\ri$-cobasic, then for all $g\in (B/C)^\vee$ such that  $\pi^{-1}({\rm ker}(g)) \not\supseteq \cA$, $\cA\cap \pi^{-1}({\rm ker}(g))$ is an $\ri$-expansion of $C$. But this implies that $\cA\cap \pi^{-1}({\rm ker}(g))= C$ for all $g\in (B/C)^\vee$ such that  $\pi^{-1}({\rm ker}(g)) \not\supseteq \cA$. By Proposition \ref{pr:postexp}\eqref{it:postexp2a}, $(A:_B \cm)\supseteq \cA$ and \eqref{it:cover1} gives that $(B/\cA\cap \pi^{-1}({\rm ker}(g)))^{\vee}$ is a cover of $(B/\cA)^\vee$. Also by Proposition \ref{pr:postexp}\eqref{it:postexp2b}, $\cA_{\ri}^B+A=\cA$. So ${\rm ker}(g) \notin \cA_{\ri}^B$. Thus $(B/\cA\cap {\rm ker}(g))^\vee$ is a non-$\ri$-cover of $(B/\cA)^\vee$.

   \eqref{it:cover3} If $C$ is a non-$\ri$-cover of $A$, then $C$ is not an $\ri$-expansion of $A$. By \eqref{it:cover1}, $B/(C\cap (g))=B/A$ there exists a $g \in (B/A)^\vee$ with $A= C \cap \pi^{-1}({\rm ker}(g))$ and $(\pi^{-1}({\rm ker}(g)):_{B} \m) \supseteq C$. Since $C$ is an $\ri$-expansion of $C$, then $C$ is an $\ri$-postexpansion of $A$ by \eqref{it:cover2}. 
    \end{proof}

\begin{prop} \label{pr:cogenpost}
 Let $(R,\cm)$ be a Noetherian complete local ring and $A$ be a strongly $\ri$-independent submodule in $B$ with $\ri$-cospread equal to $k\geq 1$ elements and $\pi: B \rightarrow B/A$ is the canonical surjection. Then any $\ri$-postexpansion of $A$ in $B$ has the form $$\bigcap\limits_{r \neq i}\pi^{-1}({\rm ker} (g_r)) \cap (\pi^{-1}({\rm ker}(g_i)):_{B} \m).$$
\end{prop}

\begin{proof}
 Let $g_1,...,g_k$ be a minimal cogenerating set for $B/A$ and let
$$\cA=\bigcap\limits_{r \neq i}\pi^{-1}({\rm ker} (g_r)) \cap (\pi^{-1}({\rm ker}(g_i)):_{(B/A)} \m).$$ Then $\cA\cap \pi^{-1}({\rm ker} (g_k))= \bigcap\limits_{i=1}^k \pi^{-1}({\rm ker}(g_i))=A$ and $(\pi^{-1}({\rm ker}(g_k)):_{B/A} \cm) \supseteq A$, so $\cA$ is a cover of $A$ by Proposition \ref{prop:cover}\eqref{it:cover1}. Also, the $g_i$ are strongly $\ri$-independent, so $A$ is $\ri$-cobasic and $A$ is the only $\ri$-expansion of itself. Thus $\cA$ is a non-$\ri$-cover of $A$. Hence, Remark \ref{prop:cover}\eqref{it:cover2} implies that $\cA$ is an $\ri$-postexpansion of $A$ in $B$.

Suppose $\cA$ is an arbitrary $\ri$-postexpansion of $A$ where $B/A=B/(\bigcap\limits_{i=1}^k \pi^{-1}({\rm ker}(g_i)))$. Then there exists an $1\leq i \leq k$ such that $\pi^{-1}({\rm ker}(g_i))\not\supseteq \cA$. Let us assume $i=k$, then $\pi^{-1}({\rm ker}(g_k))\not\supseteq \cA$. Note that since $\cA$ is an $\ri$-postexpansion of $A$ in $B$, then for any $g\in (B/A)^\vee \setminus (B/\cA)^\vee$, $\cA\cap \pi^{-1}({\rm ker} (g))$ is an $\ri$-expansion of $A$ in $B$. However, since $A$ is $\ri$-cobasic, this implies that $A=\cA\cap \pi^{-1}({\rm ker} (g))$ for any $g\in (B/A)^\vee\setminus (B/\cA)^\vee$. In particular, $A=\cA \cap \pi^{-1}({\rm ker}(g_k))$.

   Thus for $1\leq i \leq k-1$, there exists $h_i\in (B/A)^\vee$ and $b_i\in R$ such that $g_i=h_i+b_ig_k$ and 
$$\{h_1+b_1g_k,...,h_{k-1}+b_{k-1}g_k,g_k\}$$
is a minimal generating set of $(B/A)^\vee$. Thus
$$\{h_1,...,h_{k-1},g_k\}$$
is also a minimal generating set of $(B/A)^\vee$ and hence a minimal cogenerating set for $B/A$. Since $(A:_{B} \m) \supseteq \cA$ by Proposition \ref{pr:postexp}\eqref{it:postexp2a}, we have
$$\bigcap\limits_{i=1}^{k-1} \pi^{-1}({\rm ker}(h_i)) \cap (\pi^{-1}({\rm ker}(g_k)):_{B}\cm) \supseteq \bigcap\limits_{i=1}^{k-1} \pi^{-1}({\rm ker}(h_i)) \cap (A:_{B}\cm) \supseteq \cA.$$
Since $\cA$ is an $\ri$-postexpansion of $A$ in $B$ and $\bigcap\limits_{i=1}^{k-1} {\rm ker}(h_i) \cap (\pi^{-1}({\rm ker}(g_k)):_{B}\cm)$ is an $\ri$-postexpansion of $A$ in $B$, we see that $\cA=\bigcap\limits_{i=1}^{k-1} \pi^{-1}({\rm ker}(h_i)) \cap (\pi^{-1}({\rm ker}(g_k)):_{B}\cm)$.  
\end{proof}

\section{Pre-core and Post-hull}

Because minimal reductions are not unique in a Noetherian local ring, Rees and Sally defined the core of an ideal as $\textrm{core}(I)=\bigcap_{J\subset I} J$, where $J$ is an integral reduction of $I$ \cite{RS-reductions}. In this section, we will generalize the notions of $\cl$-core and $\ri$-hull, but instead of defining them in terms of $\cl$-reductions and $\ri$-expansions, we will intersect and sum the prereductions or the postexpansions for these new constructs.

\begin{defn}
\cite[Definition 2.12]{ERGV-nonres} If $(N,M)\in\cP$ the \textit{$\cl$-core} of $N$ with respect to $M$ is the intersection of all $\cl$-reductions of $N$ in $M$, or
$${\clcore{M}{N}}:=\bigcap\{L \mid L\subseteq N\subseteq L_M^{\cl} \subseteq M \textrm{ and } (L,M)\in\cP\}$$
\end{defn}

\begin{rmkx} \label{rmk:coreprered}
As long as $N$ is not $\cl$-basic in $M$, then the $\cl$-core of $N$ in $M$ will be contained in some $\cl$-prereduction of $N$ in $M$.
\end{rmkx}
\begin{proof}
Let $N$ be a submodule of $M$ that is not $\cl$-basic. Then there exists some $L\neq N$ such that $L$ is a $\cl$-reduction of $N$. Thus ${\clcore{M}{N}} = \bigcap \{ L \mid L \subseteq N \subseteq L_M^{\cl} \text{ and } (L,M) \in \cP\} \neq N$ and ${\clcore{M}{N}} \in {\bf I}_{\cl}^\prime(N,M)$. So by Proposition \ref{pr:nonreds}\eqref{it:2nonred}, there exists a $\cl$-prereduction of $N$ which contains ${\clcore{M}{N}}$.
\end{proof}

\begin{defn}
\cite[Definition 6.1]{ERGV-nonres} If $(A,B)\in\cP$, the \textit{$\ri$-hull} of a submodule $A$ with respect to $B$ is the sum of all $\ri$-expansions of $A$ in $B$, or
$${\rihull{B}{A}}:=\sum \{C\mid {\bf C}_{\ri}^B \subseteq A \subseteq C \subseteq B \textrm{ and } (C,B)\in\cP\}$$
\end{defn}

\begin{rmkx} \label{rmk:hullpostexp}
As long as $A$ is not $\ri$-cobasic in $B$, then the $\ri$-hull of $A$ in $B$ will contain some $\ri$-postexpansion of $A$ in $B$.
\end{rmkx}
\begin{proof}
Let $A$ be a submodule of $B$ that is not $\ri$-cobasic. Then there exists some $C\neq A$ such that $C$ is an $\ri$-expansion of $A$ in $B$. Thus ${\rihull{B}{A}} = \sum \{C\mid {\bf C}_{\ri}^B \subseteq A \subseteq C \subseteq B \textrm{ and } (C,B)\in\cP\} \neq A$ and ${\rihull{B}{A}}\in {\bf C}_{\ri}^\prime(A,B)$. So by Proposition \ref{pr:nonexp}\eqref{it:nonexp2}, there exists an $\ri$-postexpansion of $A$ containing ${\rihull{B}{A}}$.
\end{proof}

\begin{thm}
\cite[Theorem 6.6]{ERGV-nonres} Let $R$ be a complete Noetherian local ring. Let $A\subseteq B$ be Artinian $R$-modules and let $\ri$ be a relative Nakayama interior defined on Artinian $R$-modules. Then the $\ri$-hull of $A$ in $B$ is dual to the $\cl$-core of $(B/A)^{\vee}$ in $B^{\vee}$, where $\cl$ is the closure operation dual to $\ri$.
\end{thm}

\begin{defn}
The \textit{$\cl$-prehull of $N$ with respect to $M$} is the sum of the $\cl$-prereductions of $N$ in $M$, or
$${\clprehull{M}{N}}:=\sum \{L \mid L \textrm{ a } {\cl}\textrm{-}prereduction \textrm{ of } N \textrm{ in } M\}.$$
The \textit{$\ri$-postcore of $A$ with respect to $B$} is the intersection of the $\ri$-postexpansions of $A$ in $B$, or
$${\ripostcore{B}{A}}:=\bigcap\{C \mid  C \textrm{ a } {\ri}\textrm{-}postexpansion \textrm{ of } A \textrm{ in } B\}.$$
\end{defn}

\begin{rmkx}
$${\clprehull{M}{N}}=\sum\{L \mid L \textrm{ maximal in } {\bf I}_{\cl}^\prime(N,M)\}$$
$${\ripostcore{B}{A}}=\bigcap\{C \mid  C \textrm{ minimal in } {\bf C}_{\ri}^\prime(A,B)\}$$
\end{rmkx}
\begin{proof}
This follows immediately from the maximal elements of ${\bf I}_{\cl}^\prime(N,M)$ being $\cl$-prereductions and the minimal elements of ${\bf C}_{\ri}^\prime(A,B)$ being $\ri$-postexpanions.
\end{proof}

\begin{defn}
The \textit{$\cl$-precore} of $N$ with respect to $M$ is the intersection of all $\cl$-prereductions of $N$ in $M$,
$${\clprecore{M}{N}}=\bigcap\{L \mid L \textrm{ a $\cl$-prereduction of }  N \textrm{ in } M\}.$$
The \textit{$\ri$-posthull} of $A$ with respect to $B$ is the sum of the $\ri$-postexpansions of $A$ in $B$,
$${\riposthull{B}{A}}:=\sum\{C \mid  C \textrm{ an $\ri$-postexpansion of}  A \textrm{ in } B\}.$$
\end{defn}

The following proposition gives some upper and lower bounds on the $\cl$-core in terms of the $\cl$-precore and the $\cl$-prehull.  

\begin{prop} \label{pr:corecomp}
Let $(R,\cm)$ be a Noetherian local ring, $\cl$ a Nakayama closure, and $N$ a submodule of $M$. 
\begin{enumerate}
    \item \label{it:precorecorecomp} If $N$ is $\cl$-basic or every $\cl$-prereduction is contained in some minimal reduction, then $$\clprecore{M}{N} \subseteq \clcore{M}{N}.$$
    \item \label{it:coreprehullcomp} If $N=\clprehull{M}{N}$ or $N$ has a unique $\cl$-prereduction in $M$, then $$\clcore{M}{N} \subseteq \clprehull{M}{N}.$$
\end{enumerate}

\end{prop}

\begin{proof}
\eqref{it:precorecorecomp} First, if $N$ is $\cl$-basic, then the only $\cl$-reduction of $N$ in $M$ is $N$ itself. So ${\clcore{M}{N}}=N$. Since by definition, every $\cl$-prereduction is contained in $N$, we have ${\clprecore{N}{M}}\subseteq N$. Thus ${\clprecore{M}{N}} \subseteq \clcore{M}{N}$.

Next, suppose every $\cl$-prereduction is contained in some minimal $\cl$-reduction. Then minimal $\cl$-reductions exist and we know that $${\clcore{M}{N}}=\bigcap \{ L \mid L \text{ a minimal {\cl}-reduction of } N \text{ and } (L,M) \in \cP\}.$$
Since the intersection of $\cl$-prereductions is contained in every $\cl$-prereduction and every $\cl$-prereduction is contained in a minimal $\cl$-reduction, we see that $\clprecore{M}{N} \subseteq \clcore{M}{N}$.

\eqref{it:coreprehullcomp} If $N={\clprehull{M}{N}}=\sum \{L \mid L \textrm{ a } {\cl}\textrm{-prereduction} \textrm{ of } N \textrm{ in } M\}$. Since by definition, every $\cl$-reduction is contained in $N$, we have ${\clcore{N}{M}}\subseteq N$. Thus $\clcore{M}{N} \subseteq \clprehull{M}{N}$.

If $L$ is the unique $\cl$-prereduction of $N$ in $M$, by Remark \ref{rmk:coreprered} we know that $${\clcore{M}{N}}\subseteq L = {\clprehull{M}{N}}.$$
\end{proof}

The following proposition gives some upper and lower bounds on the $\ri$-hull in terms of the $\ri$-postcore and the $\ri$-posthull.  

\begin{prop} \label{pr:hullcomp}
Let $(R,\cm)$ be a Noetherian local ring, $\ri$ a Nakayama interior, and $A$ a submodule of $B$.
\begin{enumerate}
    \item \label{it:postcorehullcomp} If $A=\ripostcore{B}{A}$ or $A$ has a unique $\ri$-postexpansion in $B$, then $$\ripostcore{B}{A} \subseteq \rihull{B}{A}.$$ 
    \item \label{it:hullposthullcomp} If $A$ is $\ri$-cobasic or every $\ri$-postexpansions contains some maximal expansion, then $$ \rihull{B}{A} \subseteq \riposthull{B}{A}.$$
\end{enumerate}
\end{prop}

\begin{proof}
\eqref{it:postcorehullcomp} If $A={\ripostcore{B}{A}} = \bigcap\{C \mid  C \textrm{ a } {\ri}\textrm{-postexpansion} \textrm{ of } A \textrm{ in } B\}.$ Since by definition, $A$ is contained in every $\ri$-expansion and the hull is the sum of all expansions, we have $A\subseteq \rihull{B}{A}$. Thus $\ripostcore{B}{A} \subseteq \rihull{B}{A}$.

If $C$ is the unique $\ri$-postexpansion of $A$ in $B$, then by Remark \ref{rmk:hullpostexp} we know that $${\ripostcore{B}{A}}=C\subseteq {\rihull{B}{A}}.$$

\eqref{it:hullposthullcomp} First, if $A$ is $\ri$-cobasic, then the only $\ri$-expansion of $A$ in $B$ is $A$ itself. So $A=\rihull{A}{B}$. Since $A$ is contained in every $\ri$-postexpansion of $A$ in $B$, we have $A\subseteq \riposthull{B}{A}$. Thus $\rihull{B}{A} \subseteq \riposthull{B}{A}$.

Next, suppose every $\ri$-postexpansion contains some maximal $\ri$-expansion. Then maximal $\ri$-expansions exist and we know that 
$${\rihull{B}{A}}=\sum \{ C \mid C \text{ a maximal {\ri}-expansion of } A \text{ and } (C,B) \in \cP\}.$$
Since every $\ri$-postexpansion contains a maximal $\ri$-expansion, the sum of all maximal $\ri$-expansions is contained in the sum of all $\ri$-postexpansions. Thus $$ \rihull{B}{A} \subseteq \riposthull{B}{A}.$$ 
\end{proof}

We include two examples motivating the bounds given in Proposition \ref{pr:corecomp}

\begin{examplex}
 Let $R=k[[x^2,x^5]]$.  It may be helpful to refer to Example \ref{ex:prered}. 
 
 Consider the ideal $(x^4)$.  The only $\m$bf-prereduction of $(x^4)$ is $(x^6, x^9)$ by Proposition \ref{pr:clindep}.  Thus \[(x^6,x^9)=\m{\rm bf\textrm{-}precore}_R(x^4)\subsetneq \m{\rm bf\textrm{-}core}_R(x^4)=(x^4) \not\subseteq \m{\rm bf\textrm{-}prehull}_R(x^4)=(x^6,x^9)\] which gives an example that the $\cl$-precore of a submodule could be properly contained in the $\cl$-core and the $\cl$-core is not contained in the $\cl$-prehull.

 Consider the ideal $(x^4,x^7)$.  For every $a \in k$, the ideal $(x^4+ax^7)$ is a minimal $\m$bf-reductions of $(x^4,x^7)$.  Hence, $\m$bf-core$(x^4,x^7)=\bigcap\limits_{a \in k} (x^4+ax^7)=(x^6,x^9)$.  Since the ideals $I$ of $R$ contained in $(x^4,x^7)$ are of the form $(x^n+ax^{n+3})$ or $(x^n+ax^{n+3},x^{n+5})$ for $n=4$ or $n \geq 6$ and $a \in k$ or $(x^n+ax^{n+1}+bx^{n+3})$,  $(x^n+ax^{n+1}, x^{n+3})$ or  $(x^n,x^{n+1})$ for $n \geq 6$ and $a,b \in k$.  Note the only ideals $I$ which have $(x^4,x^7)$ as a cover are $(x^4+ax^7)$ or $(x^6,x^7)$.  Since $(x^4+ax^7)$ are $\m$bf-reductions of $(x^4,x^7)$, then $(x^6,x^7)$ is the only $\m$bf-prereduction of $(x^4,x^7)$ by Proposition \ref{prop:preredcov}.  Thus \[(x^6,x^7)=\m{\rm bf\textrm{-}precore}_R(x^4,x^7)\not\subseteq \m{\rm bf\textrm{-}core}_R(x^4,x^7)=(x^6,x^9) \subsetneq \m{\rm bf\textrm{-}prehull}_R(x^4,x^7)=(x^6,x^7)\] gives an example where the $\cl$-core is properly contained in the $\cl$-prehull of a submodule and the $\cl$-precore is incomparable with the $\cl$-core of a submodule.

 Consider the ideal $(x^4,x^5)$.  Note that $(x^4,x^5)^{\m{\rm bf}}=(x^4,x^5)$ and in fact is the only ideal $I$ with $I^{\m{\rm bf}}=(x^4,x^5)$.  Thus, $\m$bf-core$(x^4,x^5)=(x^4,x^5)$.  As in Example \ref{ex:prered}, the ideals $(x^4+ax^5,x^7)$ and $(x^5,x^6)$ are $\m$bf-prereductions of $(x^4,x^5)$.  Since $\bigcap\limits_{a \in k}(x^4+ax^5,x^7) \cap (x^5,x^6)=(x^6,x^7)$ and $\sum\limits_{a \in k}(x^4+ax^5,x^7) + (x^5,x^6)=(x^4,x^5)$, this gives and example where
 \[\m{\rm bf\textrm{-}precore}_R(I)\subsetneq \m{\rm bf\textrm{-}core}_R(I) = \m{\rm bf\textrm{-}prehull}_R(I)=I.\]
\end{examplex}

\begin{examplex}
    Let $k$ be a field of characteristic $p>0$, $R=k[x,y,z]/(xy,xz)$ and $\m=(x,y,z)$. Note that $I=(x+y,x+z)$ is a minimal $*$-reduction of $\m$.  Thus, by \cite[Theorem 3.10]{FVV-tightcore}, 
    \[*{\rm -core}_R(\m)=((x+y,x+z):\m)\m=\m^2.\] Now let us consider the $*$-prereductions of $\m$. By Proposition \ref{clpreredcl}, the $*$-prereductions of $\m$ are of the form $(a)+\m^{*{\rm sp}}$ where $(a,b)$ is a minimal $*$-reduction of $\m$. Note that $(x,y^2,yz,z^2)=I^{*{\rm sp}}$.  Since $(x+cy,x+dz)$ are minimal $*$-reductions for $c,d \in k$, then $(x+cy)+(x,y^2,yz,z^2)=(x,y,z^2)$ and $(x+dz)+(x,y^2,yz,z^2)=(x,y^2,z)$ are prereductions of $\m$.  Since $(x,y,z^2)+(x,y^2,z)=\m \subseteq *{\rm - prehull}_R(\m )$ we see that \[*{\rm -core}_R(\m) \subsetneq *{\rm - prehull}_R(\m ).\]  Note that $x \in I^{*{\rm sp}}$ will be in all $*$-prereductions of $\m$, thus $x \in *{\rm -precore}_R(\m) \not\subseteq  *{\rm -core}_R(\m)$.

    $I=(x+y,x+z)$ is a $*$-basic ideal.  Hence $*{\rm -core}_R(I)=I$.  Note that $(x+y,z^2)$ and $(x+z,y^2)$ are $*$-prereductions of $I$ and $*{\rm -precore}_R(I) \subseteq (x+y,z^2) \cap (x+z,y^2)=\m^2 \subseteq *{\rm -core}_R(I)=I$.
    \end{examplex}

As with the duality of the $\cl$-core and $\ri$-hull when $\ri=\cl^\dual$, we have a duality between the $\cl$-precore and the $\ri$-posthull and the $\cl$-prehull and the $\ri$-postcore.

\begin{thm}
Let $R$ be a complete Noetherian local ring. Let $A\subseteq B$ be Artinian $R$-modules, and let $\ri$ be a relative Nakayama interior defined on Artininian $R$-modules. Then the $\ri$-postcore of $A$ in $B$ is dual to the $\cl$-prehull of $(B/A)^{\vee}$ in $B^{\vee}$ and the $\ri$-posthull of $A$ in $B$ is dual to the $\cl$-precore of $(B/A)^\vee$ in $B^\vee$, where $\cl$ is the closure operation dual to $\ri$.
\end{thm}
\begin{proof}
Let $M=B^{\vee}$ and $N=(B/A)^{\vee}$. We need to show that $$(M/\clprehull{M}{N})^{\vee}=\ripostcore{B}{A}$$
and
$$(M/\clprecore{M}{N})^\vee = \riposthull{B}{A}.$$
These follows from the definitions.
\end{proof}

\begin{prop}
\cite[Proposition 7.3]{ERGV-corehull} Let $R$ be a local ring and $\cl_1\leq \cl_2$ be closure operations defined on the class of finitely generated $R$-modules $\cM$ with $\cl_2$ Nakayama. If $N\subseteq M$ are $R$-modules in $\cM$, then $\clcore[\rm cl_2]{M}{N}\subseteq \clcore[\rm cl_1]{M}{N}$.
\end{prop}

\begin{prop}
Let $(R,\cm)$ be a Noetherian local ring and $\cl_1\leq \cl_2$ Nakayama closures defined on $\cP$. If $(N,M)\in\cP$, then $\clprehull[\rm cl_2]{M}{N} \subseteq \clprehull[\rm cl_1]{M}{N}$.
\end{prop}

\begin{proof}
Let $L$ be a $\cl_2$-prereduction of $N$ in $M$. Then \[L\in {\clprehull[\rm cl_2]{M}{N}}= \sum \{L \mid L \textrm{ maximal in } {\bf I}_{\cl_2}^\prime(N,M)\}.\] By Proposition \ref{pr:compclsets}\eqref{it:compclsets1}, ${\bf I}_{\cl_2}^\prime(N,M) \subseteq {\bf I}_{\cl_1}^\prime(N,M)$ so \[\sum \{L \mid L \textrm{ maximal in } {\bf I}_{\cl_2}^\prime(N,M)\} \subseteq \sum \{L \mid L \textrm{ maximal in } {\bf I}_{\cl_1}^\prime(N,M)\} \\= \clprehull[\rm cl_1]{M}{N}.\]
\end{proof}

\begin{prop}
\cite[Proposition 7.12]{ERGV-corehull} Let $R$ be an associative (ie not necessarily commutative) ring and $\ri_1\leq \ri_2$ interior operations on a class $\cM$ of (left) $R$-modules. Let $A\subseteq B$ be $R$-modules such that $\ri_1$ and $\ri_2$ are defined on all $R$-modules between $A$ and $B$. Then $\rihull[\rm i_2]{B}{A}\subseteq \rihull[\rm i_1]{B}{A}$.
\end{prop}

\begin{prop}
Let $(R,\cm)$ be a Noetherian local ring and $\cP$ be Artinian $R$-modules. Let $\ri_1\leq \ri_2$ be Nakayama interiors on $\cP$. Then $\ripostcore[\ri_1]{B}{A}\subseteq \ripostcore[\rm i_2]{B}{A}$.
\end{prop}

\begin{proof}
Let $C\in {\ripostcore[\ri_1]{B}{A}}=\bigcap\{C \mid  C \textrm{ minimal in } {\bf C}_{\ri_1}^\prime(A,B)\}$. So $C$ is an $\ri_1$-postexpansion. By Proposition \ref{pr:comparepre}\eqref{it:comparepre1}, we know that ${\bf C}_{\ri_1}^\prime(A,B) \subseteq {\bf C}_{\ri_2}^\prime(A,B)$ so \[\bigcap\{C \mid  C \textrm{ minimal in } {\bf C}_{\ri_1}^\prime(A,B)\} \subseteq \bigcap\{C \mid  C \textrm{ minimal in } {\bf C}_{\ri_2}^\prime(A,B)\}  =\ripostcore[\ri_2]{B}{A}.\]
\end{proof}

\section{Closures $\cl$ with a special part and $\cl$-prereductions}

Vraciu first introduced the special part of tight closure in \cite{Vr-*ind} and later showed with Huneke \cite{HuVr-sp*} that in excellent normal rings with perfect residue field that the tight closure of an ideal $I$ has a special part decomposition in terms of a minimal reduction of $I$ and its special part. (See Remark \ref{rmk:sppart}(6) below.)  We will make use of this decomposition to determine all the $*$-prereductions of a tightly closed ideal. 

\begin{defn}
Let $(R, \m)$ be a Noetherian local ring of characteristic $p>0$ and $I$ be an ideal.  Let $R^o$ be the set of elements that are not in any minimal prime of $R$.  We say $x \in R$ is in the \textit{special part} of the tight closure of $I$ ($x \in I^{*{\rm sp}}$) if there exists a $c \in R^o$ such that $cx^q \in \m ^{q/q_0}I^{[q]}$ for all $q \geq q_0$. 
\end{defn}

We collect a few facts about the special part of the tight closure below:

\begin{rmk}\label{rmk:sppart}
Let $(R,\m)$ be a local ring of characteristic $p>0$ with a weak test element $c$.
\begin{enumerate}
\item An alternate description of the elements $x$ in the special part of the tight closure is:  $x \in I^{*{\rm sp}}$ if and only if there exists a $q_0$ such that $x^{q_0} \in (\m I^{[q_0]})^*$.
\item For any ideal $I$, $(I^*)^{*{\rm sp}}=I^{*{\rm sp}}=(I^{*{\rm sp}})^*$ and  if $J \subseteq I$ and $J^*=I^*$ then $J^{*{\rm sp}}=I^{*{\rm sp}}$ \cite[Lemma 3.4]{eps-spread}.
\item If $I \subseteq I^{*{\rm sp}}$ then $I$ is nilpotent.  \cite[Lemma 3.1]{eps-spread}.
\item $\m I \subseteq I^{*{\rm sp}} \cap I$ and equality holds if $I$ is a $*$-basic ideal. \cite[Lemma 3.2]{eps-spread}.
\item If $x_1, \ldots, x_k$ are $*$-independent elements in $I$, then they are also $*$-independent modulo $I^{*{\rm sp}}$. In particular, if $J=(x_1, \ldots, x_k)$ is a minimal $*$-reduction of $I$ and $x \in I^{*{\rm sp}}$, then \[J'=(x_1, \ldots, x_{i-1},x, x_{i+1}, \ldots, x_k)\] is not a minimal $*$-reduction of $I$ for any choice of $i$ \cite[Proposition 1.12]{Vr-chains}.
\item If  $R$ is further excellent and normal with perfect residue field, then for every ideal $I$ of an excellent normal ring of positive characteristic, $I^*=I+I^{*{\rm sp}}$ \cite[Theorem 2.1]{HuVr-sp*}.
\end{enumerate}
\end{rmk}

The following is a generalization of \cite[Corollary 7.2]{KRS-prered} for tight closure.

\begin{prop}
Let $(R, \m)$ be a Noetherian local ring of characteristic $p>0$ containing a weak test element.  Suppose that $I$ is a proper  principal ideal which is not nilpotent, then the unique $*$-prereduction of $I$ is $\m I$.
\end{prop}

\begin{proof}
Let $I=(x)$ where $x$ is not nilpotent.  Suppose $x$ is not $*$-independent.  Note that $(x)$ is a cover of $\m (x)$ since $(x)/(\m (x)) \cong R/\m$.  Suppose that $(x)$ is a $*$-cover of $\m (x)$.  Then $I  \subseteq I^*=(\m I)^* \subseteq (I^{* {\rm sp}})^*=I^{*{\rm sp}}$ by Remark \ref{rmk:sppart}(2,4).  However now we see that $I$ is nilpotent  by Remark \ref{rmk:sppart}(3).  Thus $x$ is $*$-independent and the set of $*$-prereductions of $I$ is $\{\m I\}$ by Proposition~\ref{prop:niltight}.  
\end{proof}

As with $*$-independence, Vraicu has noted in \cite[Observation 1.5]{Vr-chains} that if one set of generators of an ideal $K=(J,x_1, \ldots, x_k)$ are $*$-independent modulo $J$ then any set of generators of $K$ modulo $J$ are $*$-independent.  Vraciu also defined the following set in \cite{Vr-chains}.

\begin{defn}
Let $J \subseteq I$ be tightly closed ideals.  Let $\mathcal{F}(J,I)$ be the set of all tightly closed ideals $K$ such that $J \subseteq K \subseteq I$ and $\lambda(I/K)=1$.
\end{defn}

\begin{thm} \cite[Theorem 2.2]{Vr-chains}\label{thm:vrch1}
Let $(R,\m)$ be a local excellent normal ring of characteristic $p>0$ and $J \subseteq I$ be tightly closed ideals. Then $\mathcal{F}(J,I)$ is equal to the set
\[
\{(J,x_1, \ldots, x_{k-1})+I^{*{\rm sp}} \mid (J,x_1, \ldots, x_{k-1}, x_k) \text{ is a minimal } *\text{-reduction of }I \text{ modulo } J \text{ for some } x_k\}.
\]
\end{thm}

Tight closure is not the only closure which has a special part, in fact both the Frobenius closure and integral closure have a special part.  See \cite{nme-sp}.

\begin{defn}\label{def:clsp}
Let $(R, \m)$ be a Noetherian local ring and $\cl$ be a closure operation on the ideals of $R$. Then {\rm clsp} is a \textit{special part} of $\cl$ if the following four axioms hold for ideals $I \subseteq R$.
\begin{enumerate}
\item $I^{\rm clsp}$ is an ideal of $R$.
\item $\m I \subseteq I^{\rm clsp} \subseteq I^{\cl}$.
\item $(I^{\cl})^{\rm clsp}=I^{\rm clsp}=(I^{\rm clsp})^{\cl}$.
\item If $J \subseteq I \subseteq (J+I^{\rm clsp})^{\cl}$, then $I \subseteq J^{\cl}$.
\end{enumerate}
\end{defn}

Epstein showed \cite[Lemma 2.2]{nme-sp} that a closure $\cl$ with a special part is necessarily Nakayama, if $J \subseteq I$ then $J^{\rm clsp} \subseteq I^{\rm clsp}$, and if $I$ is $\cl$-independent then $\m I= I \cap I^{\rm clsp}$. 

Note that for all the known closures $\cl$ with a special part, an element $z \in I^{\rm clsp}$ if there exists some increasing function $f: \mathbb{N} \rightarrow \mathbb{N}$ and some $n \in \mathbb{N}$ such that $z^{f(n)} \in (\m I^{f(n)})^{\cl}$ or $z^{f(n)} \in (\m I^{[f(n)]})^{\cl}$. For the special part of integral closure $f(n)=n$ as $z \in I^{\rm -sp}$ if $z^n \in (\m I^n)^-$.  For the special part of tight closure and the special part of Frobenius closure, $f(n)=p^n$ where $p$ is the characteristic of the ring and $z \in I^{*{\rm sp}}$  if $z^{p^n} \in (\m I^{[p^n]})^*$ and $z \in  I^{F{\rm sp}}$ if $z^{p^n} \in (\m I^{[p^n]})^F$.   Note that in all three cases, the function $f$ defines a descending chain of ideals $I^{\{f(n)\}}\supseteq I^{\{f(m)\}}$ for  $m \geq n$, where $I^{\{f(n)\}}$ is the appropriate $f(n)$-th "power" associated to the closure $\cl$.  We will say that the special part of  $I$ with respect to the closure $\cl$ is defined by the function $f$, values $n \geq n_0 \geq 0$ and the ideals $\{I^{\{f(n)\}}\}_{n \geq n_0}$ if $z \in I^{\rm clsp}$ when 
$z^{f(n)} \in (\m I^{\{f(n)\}})^{\cl}$.  Although, we don't need this description for the proofs below, it is an interesting point that may be useful in the future.

We say that an ideal has a {\em $\cl$-special part decomposition on $R$} if $I^{\cl}=I+I^{\rm clsp}$.  Note the key ingredient to generalize Theorem~\ref{thm:vrch1} is  that strongly $\cl$-independent ideals have a $\cl$-special part decomposition.  
\begin{rmk}\label{decompdirectsum}
As in the comment after \cite[Theorem 2.1]{HuVr-sp*}, when $I$ is generated by $\cl$-independent elements and $I^{\cl}=I+I^{\rm clsp}$,  we have 
\[
\displaystyle\frac{I^{\cl}}{\m I}=\displaystyle\frac{I}{\m I} \oplus \displaystyle\frac{I^{\rm clsp}}{\m I}.
\]
\end{rmk}

First we need to generalize the following results of \cite{Vr-chains}.

\begin{prop} \label{minclsp}
(See \cite[Proposition 1.12]{Vr-chains})
Let $(R, \m)$ be a Noetherian local ring and $\cl$ a closure defined on the ideals of $R$.  Suppose that all strongly $\cl$-independent ideals $I \subseteq R$ have a $\cl$-special part decomposition.  If $x_1, \ldots, x_k$ are $\cl$-independent, then they are also $\cl$-independent modulo $I^{\rm clsp}$. In particular, if $J=(x_1, \ldots, x_k)$ is a minimal $\cl$-reduction of $I$ and $x \in I^{\rm clsp}$, then \[J'=(x_1, \ldots, x_{i-1},x, x_{i+1}, \ldots, x_k)\] is not a minimal $\cl$-reduction of $I$ for any choice of $i$.
\end{prop}

\begin{proof}
For each $i \in {1, . . . , k}$, we need to show that
$f_i \notin(I^{\rm clsp}, f_1, \ldots , f_{i-1}, f_{i+1}, \ldots, fl )^{\cl}$.
Otherwise, one could extract a minimal $\cl$-reduction $J'$ generated by
some of  the $f_j$, $ j \neq i$  and some of the elements in $I^{\rm clsp}$. Since  $f_1, \ldots, f_{i-1}, f_{i+1}, \ldots , f_k$ alone cannot generate a \cl-reduction, it follows that
we must have elements in $I^{\rm clsp}$ among the minimal generators of $J'$. But $I^{\rm clsp} \bigcap J' \subseteq \m I $, and this is a contradiction.
\end{proof}

The proof of the following Proposition is similar to the proof of \cite[Corollary 1.14]{Vr-chains}; thus, we omit the proof.

\begin{prop}\label{pr:Vrcor1.14}
    Let $(R,\m)$ be a Noetherian local ring and $J \subseteq I$ be ideals of $R$ with $I=I^{\cl}$ and assume that the $\cl$-spread of $I/J$ is $k$.  Let $g_1, \ldots g_i$ be part of a minimal system of generators of $I/J$ and let $K=(g_1, \ldots,g_i)+J$.  Then $K$ can be extended to a minimal $\cl$-reducgtion of $I/J$ if and only if $I^{\cl{\rm sp}} \cap K \subseteq \m I +J$.
\end{prop}

The following set is modeled after a set defined by Vraciu in \cite{Vr-chains} for tight closure.

\begin{defn}
Let $J \subseteq I$ be a $\cl$-closed ideals.  Let $\mathcal{F}_{\cl}(J,I)$ be the set of all $\cl$-closed ideals $K$ such that $J \subseteq K \subseteq I$ and $\lambda(I/K)=1$.
\end{defn}

The proof of the following Theorem is modelled after Vraciu's proof of Theorem~\ref{thm:vrch1}.

\begin{thm} \label{thm:vrch1cl}
Let $(R,\m)$ be a Noetherian local ring and $\cl$ a Nakayama closure operation on $R$.  Suppose  all strongly $\cl$-independent ideals $I \subseteq R$ have a $\cl$-special part decomposition and $J \subseteq I$ be $\cl$-closed ideals.  $\mathcal{F}_{\cl}(J,I)$ is equal to the set
\[
\{(J,x_1, \ldots, x_{k-1})+I^{\rm clsp} \mid (J,x_1, \ldots, x_{k-1}, x_k) \text{ is a minimal } {\cl}\text{-reduction of }I \text{ modulo } J \text{ for some } x_k\}.
\]
\end{thm}

\begin{proof}
First we show that any ideal $K$ of the form $(J,x_1, \ldots, x_{k-1})+I^{\rm clsp}$ is in $\mathcal{F}_{\cl}(J,I)$.  Since $I=(J,x_1, \ldots,x_k)+I^{\rm clsp}$ and $\m I \subseteq I^{\rm clsp}$, it follows that $\lambda(I/K)=1$ and in fact, $I/K$ is spanned by the image of $x_k$.  In order to see that $K$ is $\cl$-closed, note that we have $x_k \notin K^{\cl}$ by Proposition~\ref{minclsp}.  On the other hand $K^{\cl} \subseteq I^{\cl}=I$, and every element in $I \setminus K$ is congruent modulo $K$ to a unit multiple of $x_k$.  This shows that $K$ is $\cl$-closed.  

Conversely, we need to show that every ideal $K \subseteq \mathcal{F}_{\cl}(J,I)$ has the given form.  We need only show that there exist $x_1, \ldots, x_{k-1}$ such that $(J, x_1, \ldots, x_k)$ is a minimal $\cl$-reduction of $I$ for some choice of $x_k$, and $(Jx_1, \ldots, x_{k-1})+I^{\rm clsp} \subseteq  K$.  With this inclusion, the equality will follow as

\[
{\rm dim} \left( \displaystyle\frac{K}{\m I +J}\right)={\rm dim}\left( \displaystyle\frac{I}{\m I +J}\right)-1=k-1+\left( \displaystyle\frac{I^{\rm clsp}+J}{\m I +J}\right)
\]
and Proposition \ref{pr:Vrcor1.14} shows that this is equal to the dimension of
\[
\displaystyle\frac{(J,x_1, \ldots, x_{k-1})+I^{\rm clsp}}{\m I +J}.
\]

Choose $L=(J,x_1,\ldots, x_k)$ an arbitrary minimal $\cl$-reduction of $I$ modulo $J$.  Since $K$ is $\cl$-closed, $L \nsubseteq K$.  Since $\lambda(I/K)=1$, it follows that $K+L=I$, and therefor $\lambda(L/(K \cap L)=\lambda((K+L)/K)=1$.  Hence, we can choose generators for $L/J$ such that 
\[
(J,x_1, \ldots x_{k-1},\m x_k) \subseteq K, \quad x_k \notin K.
\]
Hence, the image of $x_k$ generates $I/K$.

We must still show that $I^{\rm clsp} \subseteq K$.  Let $z \in I^{\rm clsp}$.  We may assume that $z \notin K$ and search for a contradiction.  Since $z \notin K$, then $z \equiv ux_k {\rm mod } K$, where $u$ is a unit in $R$.  Observe that
\begin{align*}
I &= (J,x_1, \ldots, x_k)^{\cl}\\
				&= (K+ (J,x_1, \ldots, x_k))^{\cl} \text{ (Since } (J,x_1, \ldots x_{k-1},\m x_k) \subseteq K \subseteq I \text{),}  \\
				&=  (K+(z))^{\cl} \text{ (Since } z \equiv ux_k {\rm mod} K \text{),}\\
				&\subseteq (K+I^{\rm clsp})^{\cl} \text{ (Since } z \in I^{\rm clsp} \text{).}
\end{align*}
Now by Definition~\ref{def:clsp}(4), we see that $I \subseteq K^{\cl}=K$ which is a contradiction.
\end{proof}

\begin{thm}\label{clpreredcl}
Let $(R,\m)$ be  a Noetherian local ring, $\cl$ a Nakayama closure which satisfies the property that all strongly $\cl$-independent ideals have a $\cl$-special part decomposition in $R$.  Let $I$ be a $\cl$-closed ideal.  Suppose that the $\cl$-spread of $I$ is $k$.  Then the set of all $\cl$-prereductions of $I$ is
\[
{\bf P}_{\cl}(I)=\{(y_1, \ldots, y_{k-1})+I^{{\rm clsp}} \mid (y_1, \ldots, y_{k}) \text{ a minimal } \cl\text{-reduction of } I\}.
\]
\end{thm}

\begin{proof}
First we will show that if $J=(x_1, \ldots, x_k)$ is a minmal $\cl$-reduction of $I$ and $y_1, \ldots, y_k$ any minimal generating set of $J$, then $(y_1, \ldots, y_{k-1})+I^{{\rm clsp}}$ is a $\cl$-prereduction.  
Note that \[(y_1, \ldots, y_{k-1})+y_k\m\] is a $\cl$-prereduction of $J$ by Proposition \ref{prop:niltight}.  By Defintion \ref{def:clsp} (2), we have $y_k\m \subseteq \m I \subseteq I^{{\rm clsp}}$ implying that
\[(y_1, \ldots, y_{k-1})+y_k\m \subseteq (y_1, \ldots, y_{k-1})+I^{*{\rm sp}}.\] By Theorem \ref{thm:vrch1cl}, $(y_1, \ldots, y_{k-1})+I^{{\rm clsp}}$ is $\cl$-closed with \[\lambda (I/((y_1, \ldots, y_{k-1})+I^{{\rm clsp}}) =1.\]  In otherwords, $I$ is a non-$\cl$-cover of $(y_1, \ldots, y_{k-1})+I^{{\rm clsp}}$.  Since \[(y_1, \ldots, y_{k-1})+I^{{\rm clsp}}+(y_k)=I\] and $I$ is $\cl$-closed then by Proposition \ref{prop:preredcov} (2), $(y_1, \ldots, y_{k-1})+I^{{\rm clsp}}$ is a $\cl$-prereduction of $I$.

Now suppose that $\ca$ is a $\cl$-prereduction of $I$.  We need to show that $\ca=(x_1, \ldots, x_{k-1}) +I^{{\rm clsp}}$ where $(x_1, \ldots, x_k)$ is a minimal $\cl$-reduction of $I$.  First we show that a minimal $\cl$-reduction of $\ca$ has $k-1$ generators.  Since $\ca$ is a $\cl$-prereduction of $I$,  $\ca^{\cl} \subsetneq I^{\cl}=I$, but for any $x \in I \setminus \ca$, $(\ca,x)^{\cl}=I$.  Suppose $J=(x_1, \ldots, x_{\ell})$ is a minimal $\cl$-reduction of $\ca$. Since for any $x \in I \setminus \ca$, 
\[I=(\ca,x)^{\cl}=(J^{\cl}+(x))^{\cl}=(J+(x))^{\cl}=(x_1, \ldots, x_{\ell},x)^{\cl}\] then $(x_1, \ldots, x_{\ell},x)$ is a $\cl$-reduction of $I$ for any $x \in I \setminus \ca$.  Since the $\cl$-spread of $I$ is $k$, then $\ell \geq k-1$.  

Note that $(x_1, \ldots, \hat{x_i},\ldots, x_{\ell})$ is not a $\cl$-reduction of $\ca$ for any $i$.  Suppose that $\ell > k-1$ and \[(x_1, \ldots, \hat{x_i},\ldots, x_{\ell},x) \text{ is a  {\cl}-reduction of } I\] for every $x \in I \setminus \ca$ and some $i$, then $(x_1, \ldots, \hat{x_i},\ldots, x_{\ell})^{\cl}$ is a $\cl$-prereduction of $I$.  However, since $(x_1, \ldots, \hat{x_i},\ldots, x_{\ell})^{\cl} \subseteq \ca$, then it must be the case that $(x_1, \ldots, \hat{x_i},\ldots, x_{\ell})^{\cl} =\ca$ as both $(x_1, \ldots, \hat{x_i},\ldots, x_{\ell})^{\cl}$ and $\ca$ are elements of ${\bf I}'_{\cl}(I)$ and they both must be maximal in ${\bf I}'_{\cl}(I)$.  This contradicts $J$ being a minimal $\cl$-reduction of $\ca$.  Hence $\ell \leq k-1$.

Now we need to show that $I^{{\rm clsp}} \subseteq \ca$.  Suppose $x \in I^{{\rm clsp}} \setminus \ca$.  Since $\ca+(x)$ is a $\cl$-reduction of $I$ by definition of $\cl$-prereduction, then for any minimal $\cl$-reduction $(x_1, \ldots, x_{k-1})$ of $\ca$, we have that $(x_1, \ldots, x_{k-1},x)^{\cl}=(\ca+(x))^{\cl}=I$.  However, this implies that $(x_1, \ldots, x_{k-1},x)$ is a minimal $\cl$-reduction of $I$ which is a contradiction by Proposition~\ref{minclsp}.
\end{proof}

The following two corollaries follow directly from Theorem \ref{clpreredcl}.

\begin{cor}
Let $(R,\m)$ be  a local excellent normal ring of characteristic $p>0$ and $I$ be a tightly closed ideal.  Suppose that the $*$-spread of $I$ is $k$.  Then the set of all $*$-prereductions of $I$ is
\[
{\bf P}_*(I)=\{(y_1, \ldots, y_{k-1})+I^{*{\rm sp}} \mid (y_1, \ldots, y_{k}) \text{ a minimal } *\text{-reduction of } I\}.
\]
\end{cor}

\begin{cor}
Let $(R,\m)$ be  a local excellent normal ring of characteristic $p>0$ and $I$ be a Frobenius closed ideal.  Suppose that the $F$-spread of $I$ is $k$.  Then the set of all $F$-prereductions of $I$ is
\[
{\bf P}_F(I)=\{(y_1, \ldots, y_{k-1})+I^{F{\rm sp}} \mid (y_1, \ldots, y_{k}) \text{ a minimal } F\text{-reduction of } I\}.
\]
\end{cor}


\providecommand{\bysame}{\leavevmode\hbox to3em{\hrulefill}\thinspace}
\providecommand{\MR}{\relax\ifhmode\unskip\space\fi MR }
\providecommand{\MRhref}[2]{%
  \href{http://www.ams.org/mathscinet-getitem?mr=#1}{#2}
}
\providecommand{\href}[2]{#2}

\end{document}